\documentclass[10pt,reqno]{amsart}

\usepackage[utf8]{inputenc}
\usepackage[T1]{fontenc}

\usepackage{fixltx2e}
\usepackage{microtype}

\usepackage{amscd}
\usepackage{amssymb}
\usepackage{bbm}

\usepackage[letterpaper,hmargin=1in,vmargin=1in]{geometry}
\usepackage{mathtools}
\usepackage{mathrsfs}
\usepackage{enumerate}
\usepackage{ifthen}

\usepackage{notation}
\usepackage{claim}
\usepackage{bdstruct}

\usepackage{ifpdf}
\ifpdf
\usepackage[pdftex,%
pdfauthor={Marcel Kenji de Carli Silva and Levent Tunçel},%
pdftitle={Vertices of Spectrahedra arising from the Elliptope, the Theta Body, and Their Relatives},%
colorlinks=true,%
hypertexnames=true,bookmarks=true,pagebackref]{hyperref}
\else
\fi

\theoremstyle{definition}
\newtheorem{corollary}{Corollary}
\newtheorem{definition}[corollary]{Definition}
\newtheorem{proposition}[corollary]{Proposition}
\newtheorem{theorem}[corollary]{Theorem}
\newtheorem{lemma}[corollary]{Lemma}

\numberwithin{equation}{section}

\title[%
Vertices of Spectrahedra arising from the Elliptope, the Theta Body,
and Their Relatives%
]{%
  Vertices of Spectrahedra arising from the Elliptope,\\
  the Theta Body, and Their Relatives%
}

\author{Marcel K.\ de Carli Silva}
\thanks{%
  Research of the first author was supported in part by a Sinclair
  Scholarship, a Tutte Scholarship, Discovery Grants from NSERC, and
  by ONR research grant N00014-12-10049.
}

\author{Levent Tunçel}
\thanks{%
  Research of the second author was supported in part by a research
  grant from University of Waterloo, Discovery Grants from NSERC and
  by ONR research grant N00014-12-10049.
}

\newcommand*{\cb}{\bar{c}}
\newcommand*{\Xb}{\bar{X}}
\newcommand*{\xb}{\bar{x}}
\newcommand*{\Xh}{\hat{X}}
\newcommand*{\xh}{\hat{x}}
\newcommand*{\xt}{\tilde{x}}
\newcommand*{\yb}{\bar{y}}
\newcommand*{\Yh}{\hat{Y}}
\newcommand*{\zb}{\bar{z}}

\def\TH{\operatorname{TH}}

\begin{document}

\begin{abstract}
  Utilizing dual descriptions of the normal cone of convex
  optimization problems in conic form, we characterize the vertices of
  semidefinite representations arising from Lovász theta body,
  generalizations of the elliptope and related convex sets. Our
  results generalize vertex characterizations due to Laurent and
  Poljak from the 1990's. Our approach also leads us to nice
  characterizations of strict complementarity and to connections with
  some of the related literature.
\end{abstract}

\maketitle

\section{Introduction}

The study of the boundary structure of polyhedra arising from
combinatorial optimization problems has been a very successful
undertaking in the field of polyhedral combinatorics. Part of this
success relies on a very rich interplay between geometric and
algebraic properties of the faces of such polyhedra and corresponding
combinatorial structures of the problems they encode. This remains
true even in the context of some NP\nbd-hard problems, where one is
generally resigned to seek partial characterizations of the boundary
structure via some families of facets.

A different line of attack on combinatorial optimization problems,
which has become quite popular, is that of utilizing semidefinite
programming (SDP) relaxations. For the stable set problem on perfect
graphs, semidefinite formulations provide the only known approach for
efficient solution. Feasible regions of SDPs, known as
\emph{spectrahedra}, are in general much richer in complexity than
polyhedra. However, or perhaps owing to that, it is reasonable to
presume the existence of a wealth of combinatorial information encoded
in the boundary structure of spectrahedra arising from combinatorial
optimization problems. Indeed, since semidefinite optimization is a
strict generalization of linear optimization, SDPs should in principle
encode at least all that is known via polyhedral combinatorics.

Nonetheless, currently, results relating the boundary structure of
SDPs and combinatorial properties of the corresponding problems are
rather scarce. In fact, even the study of the boundary structure of
SDPs \textit{per~se} is somewhat limited. A representative sample
seems to be given by~\cite{ChristensenV79a, Loewy80a, LiT94a,
  Barvinok95a, Pataki96a, Pataki98a, Barvinok01a, Alfakih06a,
  NieRS10a, BhardwajRS11a}.

A plausible reason behind this scarcity is simple to guess. From the
viewpoint of linear conic optimization, a (pointed) polyhedron is the
intersection of the nonnegative orthant~\(\Reals_+^n\) with an affine
subspace of~\(\Reals^n\), whereas a spectrahedron is the intersection
of the positive semidefinite cone~\(\Psd{n}\) with an affine subspace
of the set~\(\Sym{n}\) of \(n \times n\) symmetric matrices. By
regarding \(\Sym{n}\) as~\(\Reals^{n(n+1)/2}\) (and thus stripping off
the extremely convenient algebraic structure of~\(\Sym{n}\)), one
could argue that nothing is gained in terms of ambient space or affine
constraints when moving from polyhedra to spectrahedra (though we
shall question this very statement later on). On the other hand, the
boundary structure of the cone \(\Psd{n}\), while completely
understood (see, e.g., \cite{WolkowiczSV00a}), is far more intricate
than that of~\(\Reals_+^n\). The latter is in fact separable in that
it may be written as the direct sum of~\(n\) copies of the nonnegative
line~\(\Reals_+\). In this context, one is comparing the rich boundary
structure of~\(\Psd{n}\) with the trivial boundary structure
of~\(\Reals_+\). This difference in complexity goes even further when
contrasting the boundary structure of spectrahedra and polyhedra,
since the intersection of an affine subspace with~\(\Psd{n}\) can be
so pathological that Strong Duality as well as Strict Complementarity
may fail for SDPs.

In this paper, we are interested in the \emph{vertices} of
spectrahedra arising from combinatorial optimization problems. We
focus on SDPs relaxations of the two combinatorial problems most
successfully attacked via SDPs, namely, \MaxCut\ and the stable set
problem. The aspects of the boundary structure we shall study revolve
around the concept of \emph{normal cone}. By carefully analyzing a
simple expression for the normal cone, we identify all vertices of
some of the spectrahedra arising from these two problems. We also
point out a simple relation between normal cones and strict
complementarity, which may be helpful in proving that the latter holds
for specific SDPs.

Vertices are naturally among the first objects to understand in a
study of the boundary structure. Recall that a \emph{vertex} of a
convex set is an \emph{extreme point} of that set whose normal cone is
full-dimensional. For a polyhedron, extreme points and vertices
coincide, and there are only finitely many of them. On the other hand,
the unit ball \(\setst{x \in \Reals^d}{\norm{x} \leq 1}\), which is
linearly isomorphic to a spectrahedron, has infinitely many extreme
points and no vertices whenever \(d \geq 2\). Indeed, the extreme
points of the unit ball~\(B\) in~\(\Reals^n\) centered at the origin
are precisely the unit vectors, but the normal cone at each such
vector is one-dimensional. This example illustrates why an extreme
point of a set whose normal cone is one-dimensional is called
\emph{smooth}, and how the dimension of a normal cone at a point is a
measure of the ``degree of non-smoothness'' of the set at that point.

Vertices of a convex set can also be regarded as the only likely
points to optimize a uniformly chosen linear function, in the
following sense. Fix a full-dimensional convex set~\(\ConvexSet \subseteq
\Reals^n\) and a point~\(\xb \in \ConvexSet\). Now choose a unit vector \(c
\in \Reals^n\) uniformly at random. Then the probability that \(\xb\)
is an optimal solution for the optimization problem \(\max
\setst{\iprod{c}{x}}{x \in \ConvexSet}\) is positive if and only if \(\xb\)
is a vertex of~\(\ConvexSet\).

The property described above may have practical significance in some
contexts where one formulates an SDP relaxation to a problem and the
vertices of the feasible region correspond exactly to the
combinatorial (or non-convex) objects from that problem. This kind of
situation may be useful in low-rank recovery schemes;
see~\cite{SaundersonCPW12a}. Other instances occur in combinatorial
optimization, in some previous results which suggest that vertices of
feasible regions of SDPs play an analogous role to that of extreme
points in polyhedral combinatorics. We discuss these next.

We start with the \emph{elliptope}~\(\Elliptope{V}\), the
spectrahedron arising from the famous SDP relaxation for \MaxCut\
utilized by Goemans and Williamson~\cite{GoemansW95a} in their
approximation algorithm. Laurent and Poljak~\cite{LaurentP95a,
  LaurentP96a} proved that all the vertices of the elliptope are
rank-one, i.e., they correspond precisely to the \emph{exact}
solutions to the \MaxCut\ problem. Next, we consider the \emph{theta
  body}~\(\TH(G)\) of a graph~\(G\), introduced
in~\cite{GroetschelLS86a} as a relaxation of the stable set polytope
of~\(G\). Shepherd~\cite{Shepherd01a} observed that, by a result
of~\cite{GroetschelLS84a}, the vertices of~\(\TH(G)\) are
precisely the incidence vectors of stable sets of~\(G\), i.e., again
the \emph{exact} solutions for the stable set problem. As far as we
know, these are the only results in the literature about vertices of
spectrahedra arising from combinatorial optimization problems.

One of our main results is both a generalization of the aforementioned
result by Laurent and Poljak and a different version of Shepherd's
observation. We describe it briefly. The theta body~\(\TH(G)\) of
a graph~\(G = (V,E)\) is naturally described as the projection
onto~\(\Reals^V\) of the feasible region of an SDP lying in the space
of symmetric matrices indexed by~\(\setlift{V}\), where \(0\) is a new
element. We denote this feasible region by \(\liftedTH(G)\). When the
graph~\(G\) has no edges, then \(\liftedTH(G)\) is a well-known
relaxation of the \emph{boolean quadric polytope} and it is linearly
isomorphic to the elliptope~\(\Elliptope{\setlift{V}}\). We will prove
that, for any graph~\(G\), all vertices of~\(\liftedTH(G)\) are
rank-one, i.e., they are the \emph{exact} solutions for the stable set
problem in the sense that they are the symmetric tensors of incidence
vectors of stable sets in~\(G\).

Using similar arguments, we shall find all vertices of some well-known
variants of \(\liftedTH(G)\) and~\(\Elliptope{V}\). These include the
SDP usually presented to introduce the Lovász theta number and its
variants, and also the SDP studied by Kleinberg and
Goemans~\cite{KleinbergG98a} for the vertex cover problem.

We should remark that throughout the paper we only study spectrahedra
in a very special form. In the literature, it is common to define
spectrahedra as sets of the form \(\setst{y \in \Reals^m}{A_0 +
  \sum_{i=1}^m y_i A_i \in \Psd{n}}\) for given matrices
\(A_0,\dotsc,A_m \in \Sym{n}\); the defining constraint is known as a
linear matrix inequality (LMI). For the sake of convenience, we shall
instead focus only on spectrahedra defined as the intersection of the
cone~\(\Psd{n}\) with an affine subspace of~\(\Sym{n}\). An advantage
is that, by confining ourselves to subsets of symmetric matrices, we
retain the ability to use directly the simple but powerful algebraic
structure of the underlying space~\(\Sym{n}\).

We start the next section with a general set-up for convex optimization
problems in conic form.  In this general form, we state and prove a dual
characterization of the normal cone.  Then we turn to the vertices of
spectrahedra arising from interesting combinatorial optimization problems.
Continuing with the normal cone, duality and boundary structure themes,
we conclude with a discussion of characterization of strict complementarity
via the normal cone and facially exposed faces of the polar convex bodies.

\section{Some Foundational Results}

\subsection{Notation and Preliminaries}

We work throughout with finite-dimensional inner-product spaces
over~\(\Reals\), and we denote them by \(\EuclideanA\)
and~\(\EuclideanB\). We denote the inner-product of \(x,y \in
\Euclidean\) by \(\iprod{x}{y}\). The \emph{dual} of~\(\EuclideanA\)
is denoted by~\(\EuclideanA^*\). The adjoint of a linear map \(\Acal
\ffrom \EuclideanA \fto \EuclideanB\) is denoted by \(\Acal^*\). If
\(\Acal\) is nonsingular, we set \(\Acal^{-*} \coloneqq
(\Acal^*)^{-1}\). If \(\Cone \subseteq \Euclidean\) is a pointed
closed convex cone with nonempty interior, the \emph{Löwner partial
  order} on~\(\Euclidean\) is defined by setting \(x \succeq_{\Cone}
y\) if \(x - y \in \Cone\).

Let \(U,V\) be finite sets. We equip the set \(\Reals^{U \times V}\)
of real \(U \times V\) matrices with the Frobenius inner-product
defined as \(\iprod{A}{B} \coloneqq \trace(A^{\transp} B)\), where
\(\trace\) is the trace. Let \(X \in \Reals^{U \times V}\). If \(S
\subseteq U\) and \(T \subseteq V\), then \(X[S,T]\) denotes the
submatrix of~\(X\) in \(\Reals^{S \times T}\). We also write \(X[S]
\coloneqq X[S,S]\).

Let \(V\) be a finite set. We denote the set of \(V \times V\)
symmetric matrices by~\(\Sym{V}\), the set of \(V \times V\) positive
semidefinite matrices by~\(\Psd{V}\), and the set of \(V \times V\)
positive definite matrices by~\(\Pd{V}\). For a positive
integer~\(n\), set \([n] \coloneqq \set{1,\dotsc,n}\). When \(V =
[n]\) we abuse the notation and write \(\Sym{n}\) for \(\Sym{[n]}\)
and similarly for other sets with a superscript~\(V\). Note that
\(\Sym{V}\) is a linear subspace of \(\Reals^{V \times V}\). For \(X
\in \Sym{n}\), we denote by \(\lambdadec(X) \in \Reals^n\) the vector
of eigenvalues of~\(X\) in non-increasing order. The map
\begin{equation*}
  \Symmetrize
  \ffrom Y \in \Reals^{V \times V}
  \mapsto \thalf\paren[\big]{Y + Y^{\transp}}
\end{equation*}
is the orthogonal projection from~\(\Reals^{V \times V}\)
onto~\(\Sym{V}\). For a matrix \(L \in \Reals^{V \times V}\), the map
\begin{equation*}
  \CongMap{L} \ffrom X \in \Sym{V} \mapsto L X L^{\transp}
\end{equation*}
is the congruence mapping. Note that \(\Symmetrize\)
and~\(\CongMap{L}\) commute.

The set of nonnegative reals is denoted by~\(\Reals_+\). Let \(V\) be
a finite set. The standard basis vectors of~\(\Reals^V\) are
\(\setst{e_i}{i \in V}\). The \emph{support} of \(x \in \Reals^V\) is
\(\supp(x) \coloneqq \setst{i \in V}{x_i \neq 0}\). The orthogonal
group on~\(V\) is denoted by~\(\OOrtho(V)\). The map \(\diag \ffrom
\Reals^{V \times V} \fto \Reals^V\) extracts the diagonal of a matrix;
its adjoint is denoted by \(\Diag\).

For a finite set~\(V\), we let \(\tbinom{V}{2}\) denote the set of all
subsets of~\(V\) of size~\(2\). If \(i,j \in V\) are distinct, we
abbreviate~\(\set{i,j}\) to~\(ij\). We also use the \emph{Iverson
  bracket}: if \(P\) is a predicate, we set
\begin{equation*}
  \Iverson{P}
  \coloneqq
  \begin{cases*}
    1 & if \(P\) holds; \\
    0 & otherwise.      \\
  \end{cases*}
\end{equation*}

Let \(\ConvexSet \subseteq \Euclidean\) be a convex set. The
\emph{relative interior} of \(\ConvexSet\) is denoted by
\(\relint(\ConvexSet)\). The \emph{boundary} of~\(\ConvexSet\) is
denoted by~\(\bd(\ConvexSet)\). The \emph{polar} of~\(\ConvexSet\) is
\(\polar{\ConvexSet} \coloneqq \setst{y \in \Euclidean^*}{\iprod{y}{x}
  \leq 1\,\forall x \in \ConvexSet}\). The smallest convex cone
containing~\(\ConvexSet\), with the origin adjoined, is denoted
by~\(\cone(\ConvexSet)\). The \emph{support function}
of~\(\ConvexSet\) is
\begin{equation*}
  \suppf{\ConvexSet}{y}
  \coloneqq
  \sup\setst*{
    \iprod{y}{x}
  }{
    x \in \ConvexSet
  }
\end{equation*}
and the \emph{gauge function} of~\(\ConvexSet\) is
\begin{equation*}
  \gauge{\ConvexSet}{x}
  \coloneqq
  \inf\setst{
    \lambda \geq 0
  }{
    x \in \lambda \ConvexSet
  }.
\end{equation*}
A convex subset \(\Face\) of~\(\ConvexSet\) is a \emph{face}
of~\(\ConvexSet\) if \(x,y \in \Face\) holds whenever \(x,y \in
\ConvexSet\) are such that the open line segment between~\(x\)
and~\(y\) meets~\(\Face\). A face~\(\Face\) of~\(\ConvexSet\) is
\emph{exposed} if it has the form \(\Face = \ConvexSet \cap
\HyperplaneA\) for a supporting hyperplane~\(\HyperplaneA\)
of~\(\ConvexSet\).

A \emph{convex corner} is a compact convex set \(\ConvexSet \subseteq
\Reals_+^V\) with nonempty interior which satisfies the following
property: if \(0 \leq y \leq x\) and \(x \in \ConvexSet\), then \(y
\in \ConvexSet\). The \emph{antiblocker} of~\(\ConvexSet\) is defined as
\begin{equation*}
  \abl{\ConvexSet}
  \coloneqq
  \polar{\ConvexSet} \cap \Reals_+^V.
\end{equation*}

Let \(\ConvexSet \subseteq \Euclidean\) be convex and let \(\xb \in
\ConvexSet\). Define the \emph{normal cone of~\(\ConvexSet\)
  at~\(\xb\)} as
\begin{equation*}
  \NormalCone{\ConvexSet}{\xb}
  \coloneqq
  \setst[\big]{
    c \in \Euclidean^*
  }{
    \iprod{c}{x}
    \leq
    \iprod{c}{\xb}
    \:\forall x \in \ConvexSet
  }.
\end{equation*}
We say that \(\xb\) is a \emph{vertex of~\(\ConvexSet\)} if
\(\dim(\NormalCone{\ConvexSet}{\xb}) = \dim(\Euclidean^*)\).

When a convex set~\(\ConvexSet\) is described as the intersection of a
polyhedron and a pointed closed convex cone with nonempty interior,
the Strong Duality Theorem for linear conic optimization yields a
simple algebraic expression for the normal cones of~\(\ConvexSet\).
This may be seen as a dual characterization of normal cones.
\begin{proposition}
  \label{prop:conic-lp-normal-cone}
  Let \(\Cone \subseteq \Euclidean\) be a pointed closed convex cone
  with nonempty interior. Let \(\Acal \ffrom \Euclidean \fto
  \Reals^p\) and \(\Bcal \ffrom \Euclidean \fto \Reals^q\) be linear
  functions. Let \(a \in \Reals^p\) and \(b \in \Reals^q\). Set
  \(\ConvexSet \coloneqq \setst{x \in \Cone}{\Acal(x) = a,\, \Bcal(x)
    \leq b}\). Suppose that \(\ConvexSet \cap \interior(\Cone) \neq
  \emptyset\). If \(\xb \in \ConvexSet\), then
  \begin{equation}
    \label{eq:conic-lp-normal-cone}
    \NormalCone{\ConvexSet}{\xb}
    =
    \Image(\Acal^*)
    +
    \setst*{
      \Bcal^*(z)
    }{
      z \in \Reals_+^q,\,
      \supp(z) \cap \supp\paren[\big]{\Bcal(\xb) - b} = \emptyset
    }
    -
    \paren[\big]{
      \Cone^* \cap \set{\xb}^{\perp}
    }.
  \end{equation}
\end{proposition}
\begin{proof}
  First we prove `\(\subseteq\)'. Let \(c \in
  \NormalCone{\ConvexSet}{\xb}\). Then \(\xb\) is an optimal solution
  for the conic programming problem
  \begin{equation*}
    \sup\setst[\big]{
      \iprod{c}{x}
    }{
      \Acal(x) = a,\,
      \Bcal(x) \leq b,\,
      x \in \Cone
    },
  \end{equation*}
  which has a restricted Slater point by assumption, i.e., there
  exists \(\xt \in \interior(\Cone)\) such that \(\Acal(\xt) = a\) and
  \(\Bcal(\xt) \leq b\). By the Strong Duality Theorem~(see, e.g.,
  \cite[Theorem~1.1]{Silva13a}), its dual
  \begin{equation*}
    \inf\setst*{
      \iprod{a}{y}
      +
      \iprod{b}{z}
    }{
      y \in \Reals^p,\,
      z \in \Reals_+^q,\,
      \Acal^*(y) + \Bcal^*(z) \succeq_{\Cone^*} c
    }
  \end{equation*}
  has an optimal solution \(\yb \oplus \zb \in \Reals^p \oplus
  \Reals_+^q\) whose slack \(\bar{s} \coloneqq \Acal^*(\yb) +
  \Bcal^*(\zb) - c \in \Cone^*\) satisfies \(\iprod{\bar{s}}{\xb} =
  0\). (Here we use the usual inner-product \(\iprod{a}{b} =
  \iprodt{a}{b}\) in the dual space.) By complementarity, we also have
  \(\iprod{\Bcal(\xb) - b}{\zb} = 0\). Together with \(\Bcal(\xb) \leq
  b\) and \(\zb \in \Reals_+^q\), this implies that \(\supp(\zb) \cap
  \supp\paren[\big]{\Bcal(\xb) - b} = \emptyset\). Since \(c =
  \Acal^*(\yb) + \Bcal^*(\zb) - \bar{s}\), we find that \(c\) lies on
  the set described by the RHS of~\eqref{eq:conic-lp-normal-cone}.

  Next we prove `\(\supseteq\)'. Let \(\bar{s} \in \Cone^* \cap
  \set{\xb}^{\perp}\), let \(\yb \in \Reals^p\) and \(\zb \in
  \Reals_+^q\) such that \(\supp(\zb) \cap
  \supp\paren[\big]{\Bcal(\xb) - b} = \emptyset\). Set \(c \coloneqq
  \Acal^*(\yb) + \Bcal^*(\zb) - \bar{s}\). If \(x \in \ConvexSet\),
  then
  \begin{equation*}
    \begin{split}
      \iprod{c}{x}
      & =
      \iprod{\Acal^*(\yb)}{x}
      +
      \iprod{\Bcal^*(\zb)}{x}
      -
      \iprod{\bar{s}}{x}
      =
      \iprod{\yb}{\Acal(x)}
      +
      \iprod{\zb}{\Bcal(x)}
      -
      \iprod{\bar{s}}{x}
      =
      \iprod{\yb}{a}
      +
      \iprod{\zb}{\Bcal(x)}
      -
      \iprod{\bar{s}}{x}
      \\
      & \leq
      \iprod{\yb}{a}
      +
      \iprod{\zb}{b}
      =
      \iprod{\yb}{a}
      +
      \iprod{\zb}{b}
      -
      \iprod{\bar{s}}{\xb}
      =
      \iprod{\yb}{\Acal(\xb)}
      +
      \iprod{\zb}{\Bcal(\xb)}
      -
      \iprod{\bar{s}}{\xb}
      \\
      & =
      \iprod{\Acal^*(\yb)}{\xb}
      +
      \iprod{\Bcal^*(\zb)}{\xb}
      -
      \iprod{\bar{s}}{\xb}
      =
      \iprod{c}{\xb}.
    \end{split}
  \end{equation*}
  Thus, \(c \in \NormalCone{\ConvexSet}{\xb}\).
\end{proof}

Now, we move back to the special case of SDP. In this setting, it is
beneficial to exploit the extra algebraic properties of the underlying
space~\(\Sym{n}\). A conspicuous extra feature is the fact that each
point in a spectrahedron, as a matrix, has a range, a nullspace, and a
rank. We shall use these concepts to massage the
identity~\eqref{eq:conic-lp-normal-cone} for the normal cone and
obtain a simple formula for its dimension.

We start by examining the rightmost term
in~\eqref{eq:conic-lp-normal-cone}, namely \(\Cone^* \cap
\set{\xb}^{\perp}\), known as the conjugate face of~\(\xb\)
in~\(\Cone^*\). When \(\Cone\) is the positive semidefinite
cone~\(\Psd{n}\), the conjugate face of a point~\(\Xb\) in~\(\Psd{n}\)
may be described as a lifted copy of a smaller semidefinite cone,
appropriately rotated via a linear automorphism of~\(\Psd{n}\) which
depends only on the range of~\(\Xb\). This allows us to associate the
dimension of the conjugate face to the rank of~\(\Xb\), as shown by
the following well-known result (for the sake of completeness, we
include a proof):
\begin{proposition}
  \label{prop:conjugate-face-sdp}
  Let \(\Xb \in \Psd{n}\). Let \(Q \in \OOrtho(n)\) such that \(X = Q
  \Diag\paren[\big]{\lambda^{\downarrow}(\Xb)} Q^{\transp}\), and set
  \(r \coloneqq \rank(\Xb)\). Then
  \begin{gather}
    \label{eq:conjugate-face-sdp}
    \Psd{n} \cap \set{\Xb}^{\perp}
    =
    Q
    \paren[\big]{
      0 \oplus \Psd{n-r}
    }
    Q^{\transp},
    \\
    \label{eq:conjugate-face-sdp-ri}
    \relint\paren[\big]{
      \Psd{n} \cap \set{\Xb}^{\perp}
    }
    =
    Q
    \paren[\big]{
      0 \oplus \Pd{n-r}
    }
    Q^{\transp},
    \\
    \label{eq:conjugate-face-sdp-dim}
    \dim\paren[\big]{
      \Psd{n} \cap \set{\Xb}^{\perp}
    }
    =
    \binom{
      \dim\paren[\big]{\Null(\Xb)} + 1
    }{
      2
    }
    \\
    \label{eq:conjugate-face-sdp-cone}
    \Psd{n} \cap \set{\Xb}^{\perp}
    =
    \cone\setst[\big]{
      \oprodsym{b}
    }{
      b \in \Null(\Xb)
    }.
  \end{gather}
\end{proposition}
\begin{proof}
  Set \(\lambda \coloneqq \lambda^{\downarrow}(\Xb)\). Let us prove
  that
  \begin{equation}
    \label{eq:conjugate-face-sdp-aux1}
    \Psd{n} \cap \set{\Diag(\lambda)}^{\perp}
    =
    0 \oplus \Psd{n-r}.
  \end{equation}
  It is clear that~`\(\supseteq\)' holds. For the reverse inclusion,
  let \(Y \in \Psd{n} \cap \set{\Diag(\lambda)}^{\perp}\). Then \(0 =
  \iprod{Y}{\Diag(\lambda)} = \iprod{\diag(Y)}{\lambda}\), which
  together with \(\diag(Y) \geq 0\) and \(\lambda \geq 0\) implies
  that \(Y_{ii} = 0\) for each \(i \in \supp(\lambda) = [r]\). Since
  \(Y \in \Psd{n}\) we find that \(Y_{ij} = 0\) for each \(i \in [r]\)
  and \(j \in [n]\), so \(Y \in 0 \oplus \Psd{n-r}\). This
  proves~\eqref{eq:conjugate-face-sdp-aux1}.

  Set \(D \coloneqq \Diag(\lambda)\) and apply the map \(\CongMap{Q} =
  \CongMap{Q}^{-*}\) to both sides
  of~\eqref{eq:conjugate-face-sdp-aux1} to obtain
  \begin{equation*}
    \begin{split}
      Q
      \paren[\big]{
        0 \oplus \Psd{n-r}
      }
      Q^{\transp}
      & =
      \CongMap{Q}\paren[\big]{
        \Psd{n}
        \cap
        \paren{
          \linspan\set{D}
        }^{\perp}
      }
      =
      \CongMap{Q}\paren{
        \Psd{n}
      }
      \cap
      \CongMap{Q}\paren[\big]{
        \paren{
          \linspan\set{D}
        }^{\perp}
      }
      \\
      & =
      \Psd{n}
      \cap
      \paren[\big]{
        \CongMap{Q}^{-*}\paren{
          \linspan\set{D}
        }
      }^{\perp}
      =
      \Psd{n}
      \cap
      \paren{
        \linspan\set{\Xb}
      }^{\perp}.
    \end{split}
  \end{equation*}
  This proves~\eqref{eq:conjugate-face-sdp}.

  To prove~\eqref{eq:conjugate-face-sdp-ri}, apply \(\relint(\cdot)\)
  to both sides of~\eqref{eq:conjugate-face-sdp} to get
  \begin{equation*}
    \relint\paren[\big]{
      \Psd{n} \cap \set{\Xb}^{\perp}
    }
    =
    \relint\paren[\big]{
      \CongMap{Q}\paren{
        0 \oplus \Psd{n-r}
      }
    }
    =
    \CongMap{Q}\paren[\big]{
      \relint\paren{
        0 \oplus \Psd{n-r}
      }
    }
    =
    Q
    \paren[\big]{
      0 \oplus \Pd{n-r}
    }
    Q^{\transp}.
  \end{equation*}
  For~\eqref{eq:conjugate-face-sdp-dim}, use the fact that the
  nonsingular map \(\CongMap{Q}\) preserves dimension:
  \begin{equation*}
    \begin{split}
      \dim\paren{
        \Psd{n} \cap \set{\Xb}^{\perp}
      }
      & =
      \dim\paren[\big]{
        \CongMap{Q}\paren{
          0 \oplus \Psd{n-r}
        }
      }
      =
      \dim\paren[\big]{
        0 \oplus \Psd{n-r}
      }
      \\
      & =
      \dim\paren[\big]{
        \Psd{n-r}
      }
      =
      \binom{n-r+1}{2}
      =
      \binom{\dim\paren[\big]{\Null(\Xb)} + 1}{2}
    \end{split}
  \end{equation*}

  Finally, we prove~\eqref{eq:conjugate-face-sdp-cone}. Let \(Y \in
  \Psd{n}\) be arbitrary, and write~\(Y\) as \(Y = \sum_{i=1}^k
  \oprodsym{h_i}\) where \(\setst{h_i}{i \in [k]} \subseteq
  \Reals^n\). Since \(Y \in \Psd{n}\), the equation \(\iprod{\Xb}{Y} =
  0\) is equivalent to \(\qform{\Xb}{h_i} = 0\) for each \(i \in
  [k]\). Since \(\Xb \in \Psd{n}\), the latter is equivalent to \(h_i
  \in \Null(\Xb)\) for each \(i \in [k]\).
\end{proof}

As the proof of Proposition~\ref{prop:conjugate-face-sdp} illustrates,
it is often helpful to restrict our attention to a specific class of
positive semidefinite matrices (e.g., diagonal matrices) for which it
is easy to prove a result, and then extend it by changing the basis,
e.g., by applying a congruence~\(\CongMap{Q}\). We now look at how
normal cones behave when we apply such transformations.

Let \(\ConvexSet \subseteq \Euclidean\) be a convex set and let \(\xb
\in \ConvexSet\). If \(T \ffrom \EuclideanA \fto \EuclideanB\) is a
linear bijection, then
\begin{equation}
  \label{eq:normal-cone-translation}
  \begin{split}
    \NormalCone[\big]{T(\ConvexSet)}{T(\xb)}
    & =
    \setst[\big]{
      c \in \EuclideanB^*
    }{
      \iprod{c}{T(x)}
      \leq
      \iprod{c}{T(\xb)}\,
      \forall x \in \ConvexSet
    }
    \\
    & =
    \setst[\big]{
      c \in \EuclideanB^*
    }{
      \iprod{T^*(c)}{x}
      \leq
      \iprod{T^*(c)}{\xb}\,
      \forall x \in \ConvexSet
    }
    \\
    & =
    \setst[\big]{
      T^{-*}(d) \in \EuclideanB^*
    }{
      \iprod{d}{x}
      \leq
      \iprod{d}{\xb}\,
      \forall x \in \ConvexSet
    }
    =
    T^{-*}\paren[\big]{
      \NormalCone{\ConvexSet}{\xb}
    }.
  \end{split}
\end{equation}

The identity~\eqref{eq:normal-cone-translation} shows that the
coordinate-free properties of normal cones remain invariant under
linear bijections. In the case of SDPs, we can say a bit more in terms
of the rank of a feasible matrix~\(\Xb\).

\begin{lemma}
  \label{lemma:invariance}
  Let \(\Acal \ffrom \Sym{n} \fto \Reals^p\) and \(\Bcal \ffrom
  \Sym{n} \fto \Reals^q\) be linear functions. Let \(a \in \Reals^p\)
  and \(b \in \Reals^q\). Let \(L \in \Reals^{n \times n}\) be
  nonsingular, and define
  \begin{gather*}
    \Acal_L
    \coloneqq
    \Acal \compose \CongMap{L}^{-1}
    \quad
    \text{and}
    \quad
    \Bcal_L
    \coloneqq
    \Bcal \compose \CongMap{L}^{-1},
    \\
    \ConvexSet
    \coloneqq
    \setst{
      X \in \Psd{n}
    }{
      \Acal(X) = a,\,
      \Bcal(X) \leq b
    },
    \\
    \ConvexSet_L
    \coloneqq
    \CongMap{L}\paren{
      \ConvexSet
    }
    =
    \setst{
      Y \in \Psd{n}
    }{
      \Acal_L(Y) = a,\,
      \Bcal_L(Y) \leq b
    }.
  \end{gather*}
  Then, for any \(\Xb \in \ConvexSet\), we have
  \begin{rmlist*}
  \item\label{list:invariance-slater} \(\ConvexSet \cap \Pd{n} \neq
    \emptyset\) if and only if \(\ConvexSet_L \cap \Pd{n} \neq \emptyset\);
  \item\label{list:invariance-cone}
    \(\NormalCone[\big]{\ConvexSet_L}{\CongMap{L}(\Xb)} =
    \CongMap{L}^{-*}\paren[\big]{\NormalCone{\ConvexSet}{\Xb}}\)
    so
    \(\dim\paren[\big]{\NormalCone[\big]{\ConvexSet_L}{\CongMap{L}(\Xb)}}
    = \dim\paren[\big]{\NormalCone{\ConvexSet}{\Xb}}\);
  \item\label{list:invariance-adjoint} \(\Image(\Acal_L^*) =
    \CongMap{L}^{-*} \paren[\big]{\Image(\Acal^*)}\) so
    \(\dim\paren[\big]{\Image(\Acal_L^*)} =
    \dim\paren[\big]{\Image(\Acal^*)}\); and analogously
    for~\(\Image(\Bcal_L^*)\);
  \item\label{list:invariance-nullspace}
    \(\Null\paren[\big]{\CongMap{L}(\Xb)} = L^{-\transp} \Null(\Xb)\)
    so \(\rank\paren[\big]{\CongMap{L}(\Xb)} = \rank(\Xb)\).
  \end{rmlist*}
\end{lemma}
\begin{proof}
  We shall use throughout the fact that the map \(\CongMap{L}\) is an
  automorphism of~\(\Psd{n}\), and in particular it is nonsingular.

  Note that \(\ConvexSet_L \cap \Pd{n} = \CongMap{L}(\ConvexSet) \cap
  \Pd{n} = \CongMap{L}\paren[\big]{\ConvexSet \cap \Pd{n}}\). This
  proves~\eqref{list:invariance-slater}.
  Statement~\eqref{list:invariance-cone} follows
  from~\eqref{eq:normal-cone-translation},
  whereas~\eqref{list:invariance-adjoint} is elementary linear
  algebra. For~\eqref{list:invariance-nullspace}, let \(h \in
  \Reals^n\) and note that \(L\Xb L^{\transp} h = 0\) is equivalent to
  \(\Xb L^{\transp} h = 0\), i.e., \(L^{\transp} h \in \Null(\Xb)\).
\end{proof}

\subsection{Vertices of the Elliptope}

We now recall some of the results from the papers~\cite{LaurentP95a,
  LaurentP96a}. We first state a slightly generalized version of a
result by Laurent and Poljak~\cite{LaurentP96a} and give a proof for
the sake of completeness.
\begin{theorem}[\cite{LaurentP96a}]
  \label{thm:normal-cone-modular-rank}
  Let \(\Acal \ffrom \Sym{n} \fto \Reals^m\) be a linear map, and let
  \(A_i \coloneqq \Acal^*(e_i)\) for each \(i \in [m]\). Let \(b \in
  \Reals^m\) such that \(\supp(b) = [m]\). Set \(\ConvexSet \coloneqq
  \setst{X \in \Psd{n}}{\Acal(X) = b}\). Suppose that \(\ConvexSet
  \cap \Pd{n} \neq \emptyset\) and \(\rank(\sum_{i=1}^m A_i) =
  \sum_{i=1}^m \rank(A_i)\). Then, for every \(\Xb \in \ConvexSet\),
  we have
  \begin{equation}
    \label{eq:dim-normal-cone-from-Null}
    \dim\paren[\big]{
      \NormalCone{\ConvexSet}{\Xb}}
    =
    \dim\paren[\big]{
      \Image(\Acal^*)
    }
    +
    \binom{
      \dim\paren[\big]{
        \Null(\Xb)
      }
      +
      1
    }{
      2
    }.
  \end{equation}
\end{theorem}

\begin{proof}
  The proof of~`\(\leq\)' in~\eqref{eq:dim-normal-cone-from-Null}
  follows from Proposition~\ref{prop:conic-lp-normal-cone}
  and~\eqref{eq:conjugate-face-sdp-dim}.

  Now we prove the reverse inequality. We shall use the fact that,
  \begin{claimeq}
    \label{eq:dim-normal-cone-from-Null-aux1}
    if \(L \in \Reals^{n \times n}\) is nonsingular, then the
    hypotheses and conclusion of the result hold if and only if they
    also hold if \(\Acal\) is replaced with \(\Acal_L \coloneqq \Acal
    \compose \CongMap{L}^{-1}\) and \(\ConvexSet\) is replaced with
    \(\ConvexSet_L \coloneqq \CongMap{L}(\ConvexSet)\).
  \end{claimeq}
  Note that \(\rank\paren[\big]{\sum_{i=1}^m \Acal_L^*(e_i)} =
  \rank\paren[\big]{\sum_{i=1}^m \Acal^*(e_i)} = \sum_{i=1}^m
  \rank\paren[\big]{ \Acal^*(e_i)} = \sum_{i=1}^m \rank\paren[\big]{
    \Acal_L^*(e_i)}\) since \(\rank(\Acal_L^*(x)) =
  \rank\paren{L^{-\transp} \Acal^*(x) L^{-1}} = \rank(\Acal^*(x))\)
  for each \(x \in \Reals^m\). Together with
  Lemma~\ref{lemma:invariance}, this
  proves~\eqref{eq:dim-normal-cone-from-Null-aux1}.

  Let us prove that
  \begin{equation}
    \label{eq:dim-normal-cone-from-Null-aux2}
    \text{
      we may assume that \(A_i A_j = 0\) whenever \(i,j \in [m]\) are
      distinct.
    }
  \end{equation}
  For each \(i \in [m]\), set \(r_i \coloneqq \rank(A_i)\) and let
  \(B_i \in \Reals^{n \times r_i}\) with full column-rank and
  \(\Image(A_i) = \Image(B_i)\). Set \(r \coloneqq \sum_{i=1}^m r_i\).
  Then the \(n \times r\) matrix
  \begin{equation*}
    B
    \coloneqq
    \begin{bmatrix}
      B_1 & \dotsm & B_m
    \end{bmatrix}
  \end{equation*}
  has full column-rank, since our hypothesis and the relation
  \(\Image\paren{\sum_{i=1}^m A_i} \subseteq \sum_{i=1}^m \Image(A_i)
  = \sum_{i=1}^m \Image(B_i) = \Image(B)\) imply that
  \begin{equation*}
    r
    =
    \textstyle\sum_{i=1}^m r_i
    =
    \rank\paren[\big]{
      \textstyle\sum_{i=1}^m
      A_i
    }
    =
    \dim\paren*{
      \Image\paren[\big]{
        \textstyle\sum_{i=1}^m
        A_i
      }
    }
    \leq
    \dim\paren*{
      \Image(B)
    }
    =
    \rank(B).
  \end{equation*}
  Thus, there exists a nonsingular \(L \in \Reals^{n \times n}\) such
  that \(LB = \sum_{k=1}^r \oprodsym{e_k}\). If \(i,j \in [m]\) are
  distinct, then \(\Image(LB_i) \perp \Image(LB_j)\) holds and so does
  \(\Image\paren{\CongMap{L}(A_i)} \perp
  \Image\paren{\CongMap{L}(A_j)}\). Thus, \(\CongMap{L}(A_i)
  \CongMap{L}(A_j) = 0\) whenever \(i,j \in [m]\) are distinct. Thus,
  by replacing \(\Acal\) with \(\Acal \compose
  \CongMap{L^{-\transp}}^{-1}\) and
  applying~\eqref{eq:dim-normal-cone-from-Null-aux1}, this
  proves~\eqref{eq:dim-normal-cone-from-Null-aux2}.

  Next we shall refine~\eqref{eq:dim-normal-cone-from-Null-aux2}
  and show that
  \begin{claimeq}
    \label{eq:dim-normal-cone-from-Null-aux3}
    we may assume that \(A_i = \Diag(a_i)\) for each \(i \in [m]\),
    where \(a_1, \dotsc, a_m \in \Reals^n\) are vectors with pairwise
    disjoint support.
  \end{claimeq}
  Since the matrices \(A_1, \dotsc, A_m\) pairwise commute
  by~\eqref{eq:dim-normal-cone-from-Null-aux2}, there exists \(P \in
  \OOrtho(n)\) such that \(P^{\transp} A_i P\) is diagonal for each \(i
  \in [m]\). Let \(a_i \in \Reals^n\) such that \(A_i = P \Diag(a_i)
  P^{\transp}\) for each \(i \in [m]\). For distinct \(i, j \in [m]\),
  we have \(0 = P^{\transp} (A_i A_j) P = \Diag(a_i) \Diag(a_j)\),
  whence \(\supp(a_i) \cap \supp(a_j) = \emptyset\). Thus, by
  replacing \(\Acal\) with \(\Acal \compose \CongMap{P^{-1}}^{-1}\)
  and applying~\eqref{eq:dim-normal-cone-from-Null-aux1}, this
  proves~\eqref{eq:dim-normal-cone-from-Null-aux3}.

  Let \(\Xb \in \ConvexSet\), let \(\set{R_1,\dotsc,R_p}\) be a basis
  of \(\Sym{n-\rank(\Xb)}\), and let \(Q \in \OOrtho(n)\) such that
  \(\Xb = Q \Diag(\lambda) Q^{\transp}\), where \(\lambda \coloneqq
  \lambda^{\downarrow}(\Xb)\). To prove~`\(\geq\)'
  in~\eqref{eq:dim-normal-cone-from-Null}, it suffices by
  Proposition~\ref{prop:conic-lp-normal-cone} to show that the set of
  matrices \(\set{A_1, \dotsc, A_m} \cup \set{ Q (0 \oplus R_1)
    Q^{\transp}, \dotsc, Q (0 \oplus R_p) Q^{\transp}}\) is linearly
  independent. Let \(\alpha \in \Reals^m\) and \(\beta \in \Reals^p\)
  such that
  \begin{equation}
    \label{eq:dim-normal-cone-from-Null-aux4}
    \sum_{i=1}^m \alpha_i A_i
    +
    \sum_{j=1}^p \beta_j Q (0 \oplus R_j) Q^{\transp}
    =
    0.
  \end{equation}

  Let \(u \in \Null(\Xb)^{\perp}\). Then \(Q^{\transp} u \in
  Q^{\transp}\Image(\Xb) = \Image(Q^{\transp} \Xb) =
  \Image\paren[\big]{\Diag(\lambda)Q^{\transp}} \subseteq
  \Image\paren[\big]{\Diag(\lambda)}\), whence \(\supp(Q^{\transp} u)
  \subseteq [\rank(\Xb)]\). Thus, if we
  multiply~\eqref{eq:dim-normal-cone-from-Null-aux4} on the right
  by~\(u\) we obtain \(u \in \Null\paren[\big]{\sum_{i=1}^m \alpha_i
    A_i}\). So \(\Null(\Xb)^{\perp} \subseteq
  \Null\paren[\big]{\sum_{i=1}^m \alpha_i A_i}\), or equivalently
  \begin{equation*}
    \Image\paren[\big]{\textstyle \sum_{i=1}^m \alpha_i A_i}
    \subseteq
    \Null(\Xb).
  \end{equation*}
  Let \(k \in [m]\). Then
  from~\eqref{eq:dim-normal-cone-from-Null-aux3} we have
  \(\Image(\alpha_k A_k) \subseteq \Image(\sum_{i=1}^m \alpha_i A_i)
  \subseteq \Null(\Xb)\) so \(\alpha_k \Xb A_k = 0\). Since \(0 \neq
  b_k = \iprod{A_k}{\Xb} = \trace(A_k \Xb)\), we have \(\Xb A_k \neq
  0\), so it must be the case that \(\alpha_k = 0\). This proves that
  \(\alpha = 0\), whence \(\beta = 0\). This concludes the proof
  of~\eqref{eq:dim-normal-cone-from-Null}.
\end{proof}

Let \(V\) be a finite set. The set
\begin{equation*}
  \Elliptope{V}
  \coloneqq
  \setst[\big]{
    X \in \Psd{V}
  }{
    \diag(X) = \ones
  },
\end{equation*}
known as the \emph{elliptope}, is a well-known relaxation for the
\emph{cut polytope} \(\conv\setst[\big]{\oprodsym{x}}{x \in
  \set{\pm1}^V}\). The SDP used by Goemans and
Williamson~\cite{GoemansW95a} in their celebrated approximation
algorithm for \MaxCut\ has \(\Elliptope{V}\) as its feasible region
when applied to a graph on~\(V\). When we apply
Theorem~\ref{thm:normal-cone-modular-rank} to the
elliptope~\(\Elliptope{V}\), we find that a point \(\Xb\)
of~\(\Elliptope{V}\) is a vertex of~\(\Elliptope{V}\) precisely when
\(\dim\paren[\big]{\Null(\Xb)} = \card{V} - 1\):
\begin{corollary}[\cite{LaurentP95a}]
  \label{cor:elliptope-vertices}
  Let \(V\) be a finite set. Then a point \(\Xb\) of~\(\Elliptope{V}\)
  is a vertex of~\(\Elliptope{V}\) if and only if \(\rank(\Xb) = 1\).
  Thus, the vertices of~\(\Elliptope{V}\) are precisely the matrices
  of the form \(\oprodsym{x}\) with \(x \in \set{\pm1}^V\).
\end{corollary}

In the proof of Corollary~\ref{cor:elliptope-vertices} by Laurent and
Poljak~\cite{LaurentP95a}, the fact that \(\oprodsym{\xb}\) is a
vertex of~\(\Elliptope{n}\) if \(\xb \in \set{\pm1}^n\) follows from
the simple observation that \(\setst{(-1)^{\Iverson{\xb_i \xb_j <
      0}}\Symmetrize(\oprod{e_i}{e_j})}{i,j \in [n]} \subseteq
\NormalCone{\Elliptope{n}}{\oprodsym{\xb}}\). For the proof that all
vertices of~\(\Elliptope{n}\) have rank one, Laurent and Poljak give
the following argument, which we include for the sake of completeness:
\begin{proposition}[\cite{LaurentP95a}]
  \label{prop:free-offdiag-rank-bound}
  Let \(\Acal \ffrom \Sym{n} \fto \Reals^m\) and \(b \in \Reals^m\)
  and set \(\ConvexSet \coloneqq \setst{X \in \Psd{n}}{\Acal(X) =
    b}\). Suppose that \(\ConvexSet \cap \Pd{n} \neq \emptyset\).
  Suppose that for some \(k \in [n-1]\) there exists a linearly
  independent subset \(\set{h_0} \cup \setst{h_i}{i \in [k]}\)
  of~\(\Reals^n\) such that
  \(\linspan\setst{\Symmetrize(\oprod{h_0}{h_i})}{i \in [k]} \subseteq
  \Null(\Acal)\). Then every vertex of~\(\ConvexSet\) has rank \(\leq
  n - k\).
\end{proposition}

\begin{proof}
  We first show that,
  \begin{claimeq}
    \label{eq:free-offdiag-rank-bound-aux1}
    if \(L \in \Reals^{n \times n}\) is nonsingular, then the
    hypotheses and conclusion of the result hold if and only if they
    also hold if \(\Acal\) is replaced with \(\Acal_L \coloneqq \Acal
    \compose \CongMap{L}^{-1}\) and \(\ConvexSet\) is replaced with
    \(\ConvexSet_L \coloneqq \CongMap{L}(\ConvexSet)\).
  \end{claimeq}
  Note that linear independence of \(\set{h_0} \cup \setst{h_i}{i \in
    [k]} \subseteq \Reals^n\) is equivalent to that of \(\set{L h_0}
  \cup \setst{L h_i}{i \in [k]}\), and the inclusion
  \(\linspan\setst{\Symmetrize(\oprod{h_0}{h_i})}{i \in [k]} \subseteq
  \Null(\Acal)\) is equivalent to \(\linspan\setst{\Symmetrize(L
    \oprod{h_0}{h_i} L^{\transp})}{i \in [k]} =
  \CongMap{L}\paren[\big]{
    \linspan\setst{\Symmetrize(\oprod{h_0}{h_i})}{i \in [k]}}
  \subseteq \Null(\Acal \compose \CongMap{L}^{-1}) = \Null(\Acal_L)\).
  The proof of~\eqref{eq:free-offdiag-rank-bound-aux1} follows from
  these facts together with Lemma~\ref{lemma:invariance}.

  By applying~\eqref{eq:free-offdiag-rank-bound-aux1} with \(L \in
  \Reals^{n \times n}\) nonsingular such that \(L h_0 = e_n\) and \(L
  h_i = e_i\) for \(i \in [k]\),
  \begin{equation*}
    \text{
      we may assume that \(h_0 = e_n\) and
      \(h_i = e_i\) for all \(i \in [k]\).
    }
  \end{equation*}
  Set \(d \coloneqq \dim(\Null(\Acal))\). Let \(\Proj{\Null(\Acal)}
  \ffrom \Sym{n} \fto \Sym{n}\) denote the orthogonal projection
  onto~\(\Null(\Acal)\). Since the elements of
  \(\setst{\Symmetrize(\oprod{e_n}{e_i})}{i \in [k]} \subseteq
  \Null(\Acal)\) are pairwise orthogonal, we have
  \begin{equation}
    \label{eq:free-offdiag-rank-bound-aux2}
    \Proj{\Null(\Acal)}\paren[\big]{
      \Symmetrize(\oprod{e_n}{e_i})
    }
    =
    \Symmetrize(\oprod{e_n}{e_i})
    \qquad
    \forall i \in [k]
  \end{equation}
  and
  \begin{equation}
    \label{eq:free-offdiag-rank-bound-aux3}
    \text{
      there is a linear isomorphism \(\varphi \ffrom \Null(\Acal) \fto
      \Reals^{d}\) such that \(\sqbrac{\varphi(X)}_i = X_{in}\) for
      all \(i \in [k]\).
    }
  \end{equation}

  Let \(\Xb\) be a vertex of~\(\ConvexSet\). By
  Proposition~\ref{prop:conic-lp-normal-cone}
  and~\eqref{eq:conjugate-face-sdp-cone}, we have
  \begin{equation*}
    \NormalCone{\ConvexSet}{\Xb}
    =
    \Image(\Acal^*) - \paren{\Psd{n} \cap \set{\Xb}^{\perp}}
    =
    \Image(\Acal^*) - \cone\setst[\big]{\oprodsym{b}}{b \in \Null(\Xb)}.
  \end{equation*}
  Then
  \begin{equation*}
    \Proj{\Null(\Acal)}\paren[\big]{
      \NormalCone{\ConvexSet}{\Xb}
    }
    =
    -\cone\setst[\big]{
      \Proj{\Null(\Acal)}\paren{
        \oprodsym{b}
      }
    }{
      b \in \Null(\Xb)
    }
  \end{equation*}
  has dimension \(d\). Hence, there exists a set \(\setst{b_j}{j \in
    [d]} \subseteq \Null(\Xb)\) such that, if we define \(B_j
  \coloneqq \oprodsym{b_j}\) for \(j \in [d]\), the set
  \(\setst{\Proj{\Null(\Acal)}\paren{B_j}}{j \in [d]}\) is linearly
  independent. So the \(d \times d\) matrix~\(M\) whose \(j\)th column
  is \(\varphi\paren{\Proj{\Null(\Acal)}\paren{B_j}}\) is nonsingular,
  and its submatrix \(M_1 \coloneqq M[[k],[d]]\) has \(k\) linearly
  independent columns. By possibly relabeling the~\(B_j\)'s, we may
  assume that the first \(k\) columns of~\(M_1\) are linearly
  independent, i.e.,
  \begin{equation*}
    \setst[\big]{
      \sqbrac{b_j}_{n}
      \paren[\big]{
        b_j \restriction_{[k]}
      }
    }{
      j \in [k]
    }
    =
    \setst[\big]{
      \varphi\paren{
        \Proj{\Null(\Acal)}\paren{
          B_j
        }
      }
      \restriction_{[k]}
    }{
      j \in [k]
    }
    \text{ is linearly independent},
  \end{equation*}
  where the first equation follows
  from~\eqref{eq:free-offdiag-rank-bound-aux2}
  and~\eqref{eq:free-offdiag-rank-bound-aux3} since
  \(\sqbrac{\varphi\paren{\Proj{\Null(\Acal)}(B_j)}}_i
  = \paren{\Proj{\Null(\Acal)}(B_j)}_{in} = \sqbrac{B_j}_{in} =
  \sqbrac{b_j}_n \sqbrac{b_j}_i\) for each \(i \in [k]\). In
  particular, \(\sqbrac{b_j}_n \neq 0\) for each \(j \in [k]\) and
  \(\setst{b_j}{j \in [k]}\) is linearly independent. Since \(b_j \in
  \Null(\Xb)\) for each \(j \in [k]\), we get \(\rank(\Xb) \leq n -
  k\).
\end{proof}

When Proposition~\ref{prop:free-offdiag-rank-bound} is applied
to~\(\Elliptope{n}\) in the proof of
Corollary~\ref{cor:elliptope-vertices} with \(h_0 \coloneqq e_n\) and
\(h_i \coloneqq e_i\) for each \(i \in [n-1]\), we find again that
each vertex of~\(\Elliptope{n}\) is rank-one. However, the bound
provided by Proposition~\ref{prop:free-offdiag-rank-bound} may be
quite weak: the set \(\ConvexSet \coloneqq \setst{X \in
  \Psd{n}}{\iprod{\Symmetrize(\oprod{e_i}{e_j})}{X} = 0\,\forall ij
  \in \tbinom{[n]}{2}}\) has a unique vertex (and extreme point) and
its rank is~\(0\), whereas the best upper bound that
Proposition~\ref{prop:free-offdiag-rank-bound} yields in this case is
\(n-1\). Still, with the same application as for the elliptope,
Proposition~\ref{prop:free-offdiag-rank-bound} yields the following
unexpected fact:
\begin{corollary}
  \label{cor:smash-to-rank-one}
  Let \(\Acal \ffrom \Sym{n} \fto \Reals^m\) be a linear map. Let \(b
  \in \Reals^m\). Define \(\ConvexSet \coloneqq \setst{X \in
    \Psd{n}}{\Acal(X) = b}\). Suppose that \(\ConvexSet \cap \Pd{n} \neq
  \emptyset\). Then every vertex of \(\setst[\big]{\Xh \in
    \Psd{\setlift{[n]}}}{\Xh[[n]] \in \ConvexSet}\) has rank one.
\end{corollary}

\subsection{Vertices of the Theta Body}

We briefly recall some basic results about the theta body of a graph
and its lifted version, with emphasis on the boundary structure of the
former in the setting of antiblocking duality.

Let \(V\) be a finite set. Define the map
\begin{equation}
  \label{eq:bool-quad-map-def1}
  \BoolQuadMap{\setlift{V}}
  \ffrom
  \Xh \in \Sym{\setlift{V}}
  \mapsto
  \BoolQuadMap{\set{0}}(\Xh)
  \oplus
  \BoolQuadMap{V}(\Xh)
  \in
  \Reals^{\set{0}}
  \oplus
  \Reals^V
\end{equation}
where \(\BoolQuadMap{\set{0}} \ffrom \Sym{\setlift{V}} \fto
\Reals^{\set{0}}\) and \(\BoolQuadMap{V} \ffrom \Sym{\setlift{V}} \fto
\Reals^V\) are defined by
\begin{equation}
  \label{eq:bool-quad-map-def2}
  \BoolQuadMap{\set{0}}^*(e_0)
  \coloneqq
  \oprodsym{e_0}
  \qquad
  \text{and}
  \qquad
  \BoolQuadMap{V}^*(e_i)
  \coloneqq
  \Symmetrize\paren[\big]{
    \oprod{
      e_i
    }{
      (e_i - e_0)
    }
  }
  \quad
  \forall i \in V.
\end{equation}

Let \(G = (V,E)\) be a graph. The \emph{lifted theta body of~\(G\)} is
defined as
\begin{equation}
  \label{eq:def-lifted-theta-body}
  \liftedTH(G)
  \coloneqq
  \setst[\big]{
    \Xh \in \Psd{\setlift{V}}
  }{
    \BoolQuadMap{\setlift{V}}(\Xh) = 1 \oplus 0,\,
    \Acal_E(\Xh[V]) = 0
  },
\end{equation}
where \(\Acal_E \ffrom \Sym{V} \fto \Reals^E\) is defined by
\begin{equation}
  \label{eq:edge-entries-def}
  \Acal_E^*(e_{ij}) \coloneqq \Symmetrize(\oprod{e_i}{e_j})
  \qquad
  \forall ij \in E.
\end{equation}
The \emph{theta body of~\(G\)}, first introduced
in~\cite{GroetschelLS86a}, is the projection
\begin{equation*}
  \TH(G)
  \coloneqq
  \setst*{
    \diag\paren[\big]{
      \Xh[V]
    }
    \in \Reals^V
  }{
    \Xh \in \liftedTH(G)
  }.
\end{equation*}

The facets of~\(\TH(G)\) are defined precisely by the inequalities
\(x_i \geq 0\) for each \(i \in V\), and by the clique inequalities
\(\iprod{\incidvector{K}}{x} \leq 1\) for each clique~\(K\) of~\(G\);
see, e.g., \cite[Theorem~67.13]{Schrijver03b}. Consequently, as
mentioned by Shepherd~\cite{Shepherd01a}, the vertices of~\(\TH(G)\)
are precisely the incidence vectors of stable sets of~\(G\). This
follows from the formula \(\abl{\TH(G)} = \TH(\overline{G})\)
(see~\cite[Theorem~67.12]{Schrijver03b} for a proof) and the simple
duality correspondence between facets and vertices in antiblocking
pairs of convex corners, which may be stated as follows:
\begin{theorem}
  \label{thm:abl-facet-vertex}
  Let \(\ConvexSet \subseteq \Reals_+^V\) be a convex corner. Then:
  \begin{enumerate}[(i)]
  \item the facets of~\(\ConvexSet\) are determined precisely by the
    inequalities \(x_i \geq 0\), for each \(i \in V\), and
    \(\iprod{\bar{y}}{x} \leq 1\), where \(\bar{y}\) ranges over all
    vertices of \(\abl{\ConvexSet}\) such that
    \(\NormalCone{\Cscr}{\bar{y}} \cap \Reals_+^V\) is
    full-dimensional.
  \item if \(x\) is a vertex of~\(\Cscr\) with support~\(S\), then
    \(x_S \coloneqq x\restriction_S\) is a vertex of \(\Cscr_S
    \coloneqq \setst{z \in \Reals_+^S}{z \oplus 0 \in \Cscr}\) and
    \(\NormalCone{\Cscr_S}{x_S} \subseteq \Reals_+^S\).
  \end{enumerate}
\end{theorem}

Thus, as in the case of the elliptope via
Corollary~\ref{cor:elliptope-vertices}, the vertices of the theta
body~\(\TH(G)\) are the exact solutions for the problem for which
\(\TH(G)\) yields a relaxation. Our main result is the
corresponding statement in matrix space, i.e., for \(\liftedTH(G)\).
One may argue that the set \(\TH(G)\), which lives
in~\(\Reals^V\), is more natural to study, and indeed this is a good
argument if we only consider \(\TH(G)\) as a relaxation of the
stable set polytope. However, when one actually needs to optimize a
linear function over~\(\TH(G)\), the latter set is represented as
a projection of~\(\liftedTH(G)\), and the optimization takes place in
the ambient space~\(\Sym{\setlift{V}}\) of~\(\liftedTH(G)\).
Qualitatively, \(\liftedTH(G)\) is a spectrahedron, whereas
\(\TH(G)\) is the projection of a spectrahedron. These classes of
sets have quite different properties in general. We refer the reader
to the paper~\cite{ChuaT08a} for more details.

\subsection{Our Main Tool}

Now we prove the principal tool for our main result: a simple
algebraic expression for the dimension of the normal cone:
\begin{theorem}
  \label{thm:dim-normal-cone-sdp}
  Let \(\Acal \ffrom \Sym{n} \fto \Reals^p\) and \(\Bcal \ffrom
  \Sym{n} \fto \Reals^q\) be linear functions. Let \(a \in \Reals^p\)
  and \(b \in \Reals^q\). Set \(\ConvexSet \coloneqq \setst{X \in
    \Psd{n}}{\Acal(X) = a,\, \Bcal(X) \leq b}\). Suppose that
  \(\ConvexSet \cap \Pd{n} \neq \emptyset\). Let \(\Xb \in
  \ConvexSet\), and let \(P\) denote the orthogonal projection onto
  \(\setst{z \in \Reals^q}{\supp(z) \cap \supp\paren{\Bcal(\Xb) - b} =
    \emptyset}\). Then
  \begin{equation}
    \label{eq:dim-normal-cone-sdp}
    \dim\paren*{
      \NormalCone{\ConvexSet}{\Xb}
    }
    =
    \dim(\Sym{n})
    -
    \dim\paren[\Big]{
      \Null(\Acal)
      \cap
      \Null(P \compose \Bcal)
      \cap
      \linspan\setst*{
        \Symmetrize(\Xb \oprod{u}{v})
      }{
        u, v \in \Reals^n
      }
    }.
  \end{equation}
  In particular, if \(\Xb = \oprodsym{\xb}\) for some nonzero \(\xb
  \in \Reals^n\), then
  \begin{equation}
    \label{eq:dim-normal-cone-sdp-rank-one}
    \dim\paren[\big]{
      \NormalCone{\ConvexSet}{\oprodsym{\xb}}
    }
    =
    \dim(\Sym{n})
    -
    \dim\paren[\Big]{
      \paren*{
        \setst{
          A_i \xb
        }{
          i \in [p]
        }
        \cup
        \setst*{
          B_i \xb
        }{
          i \in [q] \drop \supp\paren*{\Bcal(\oprodsym{\xb}) - b}
        }
      }^{\perp}
    },
  \end{equation}
  where \(A_i \coloneqq \Acal^*(e_i)\) for all \(i \in [p]\) and \(B_i
  \coloneqq \Bcal^*(e_i)\) for all \(i \in [q]\); thus,
  \begin{claimeq}
    \(\oprodsym{\xb} \in \ConvexSet\) is a vertex of~\(\ConvexSet\) if
    and only if \(\setst{A_i \xb}{i \in [p]} \cup \setst{B_i \xb}{i
      \in [q] \drop \supp\paren{\Bcal(\oprodsym{\xb}) - b}}\)
    spans~\(\Reals^n\).
  \end{claimeq}
\end{theorem}

\begin{proof}
  We start by proving that
  \begin{equation}
    \label{eq:dim-normal-cone-sdp-aux1}
    \sqbrac[\big]{
      \linspan
      \paren[\big]{
        \Psd{n}
        \cap
        \set{\Xb}^{\perp}
      }
    }^{\perp}
    =
    \linspan\setst*{
      \Symmetrize(\Xb \oprod{u}{v})
    }{
      u, v \in \Reals^n
    }.
  \end{equation}
  Let \(Q \in \OOrtho(n)\) such that \(\Xb = Q \Diag(\lambda)
  Q^{\transp}\), where \(\lambda \coloneqq
  \lambda^{\downarrow}(\Xb)\). Set \(D \coloneqq \Diag(\lambda)\) and
  \(r \coloneqq \rank(\Xb)\). Note that
  \begin{equation}
    \label{eq:conjugate-face-sdp-perp-aux1}
    \sqbrac[\big]{
      \linspan
      \paren[\big]{
        \Psd{n}
        \cap
        \set{D}^{\perp}
      }
    }^{\perp}
    =
    \linspan\setst*{
      \Symmetrize(D \oprod{u}{v})
    }{
      u, v \in \Reals^n
    }
  \end{equation}
  since by Proposition~\ref{prop:conjugate-face-sdp} we have
  \begin{equation*}
    \begin{split}
      \sqbrac[\big]{
        \linspan\paren[\big]{
          \Psd{n}
          \cap
          \set{D}^{\perp}
        }
      }^{\perp}
      & =
      \sqbrac[\big]{
        \linspan\paren[\big]{
          0 \oplus \Psd{n-r}
        }
      }^{\perp}
      =
      \sqbrac[\big]{
        0 \oplus \Sym{n-r}
      }^{\perp}
      \\
      & =
      \linspan\setst[\big]{
        \Symmetrize\paren{\oprod{e_i}{e_j}}
      }{
        i \in [r],\, j \in [n]
      }
      \\
      & =
      \linspan\setst[\big]{
        \Symmetrize\paren{D \oprod{u}{v}}
      }{
        u,v \in \Reals^n
      }.
    \end{split}
  \end{equation*}
  In the latter equality, the inclusion~`\(\subseteq\)' is obvious.
  For the reverse inclusion, let \(u,v \in \Reals^n\) and note that
  \(\Symmetrize(D \oprod{u}{v}) = \sum_{i=1}^n \sum_{j=1}^n u_i v_j
  \Symmetrize(D\oprod{e_i}{e_j}) = \sum_{i=1}^r \sum_{j=1}^n u_i v_j
  \Symmetrize(D\oprod{e_i}{e_j}) \). This
  proves~\eqref{eq:conjugate-face-sdp-perp-aux1}.

  To prove~\eqref{eq:dim-normal-cone-sdp-aux1}, apply \(\CongMap{Q}\)
  to both sides of~\eqref{eq:conjugate-face-sdp-perp-aux1} to get
  \begin{equation*}
    \begin{split}
      \sqbrac[\big]{
        \linspan
        \paren[\big]{
          \Psd{n}
          \cap
          \set{\Xb}^{\perp}
        }
      }^{\perp}
      & =
      \sqbrac[\big]{
        \linspan
        \paren[\big]{
          \Psd{n}
          \cap
          \set{\CongMap{Q}(D)}^{\perp}
        }
      }^{\perp}
      =
      \sqbrac[\big]{
        \linspan
        \paren[\big]{
          \Psd{n}
          \cap
          \CongMap{Q}^{-*}\paren{
            \set{D}^{\perp}
          }
        }
      }^{\perp}
      \\
      & =
      \sqbrac[\big]{
        \linspan\paren[\big]{
          \CongMap{Q}^{-*}\paren{
            \Psd{n}
            \cap
            \set{D}^{\perp}
          }
        }
      }^{\perp}
      =
      \sqbrac[\big]{
        \CongMap{Q}^{-*}\paren[\big]{
          \linspan\paren{
            \Psd{n}
            \cap
            \set{D}^{\perp}
          }
        }
      }^{\perp}
      \\
      & =
      \CongMap{Q}\paren[\big]{
        \paren[\big]{
          \linspan\paren{
            \Psd{n}
            \cap
            \set{D}^{\perp}
          }
        }^{\perp}
      }
      \\
      & =
      \CongMap{Q}\paren[\big]{
        \linspan\setst[\big]{
          \Symmetrize(D \oprod{u}{v})
        }{
          u,v \in \Reals^n
        }
      }
      \\
      & =
      \linspan\setst[\big]{
        \CongMap{Q}\paren[\big]{
          \Symmetrize(D \oprod{u}{v})
        }
      }{
        u,v \in \Reals^n
      }
      \\
      & =
      \linspan\setst[\big]{
        \Symmetrize\paren[\big]{
          \CongMap{Q}(D \oprod{u}{v})
        }
      }{
        u,v \in \Reals^n
      }
      \\
      & =
      \linspan\setst[\big]{
        \Symmetrize\paren[\big]{
          \Xb \oprod{u}{v}
        }
      }{
        u,v \in \Reals^n
      }.
    \end{split}
  \end{equation*}

  By Proposition~\ref{prop:conic-lp-normal-cone}
  and~\eqref{eq:dim-normal-cone-sdp-aux1}, we have
  \begin{equation*}
    \begin{split}
      \paren[\Big]{
        \linspan\paren[\big]{
          \NormalCone{\ConvexSet}{\Xb}
        }
      }^{\perp}
      & =
      \paren[\Big]{
        \Image(\Acal^*)
        +
        \Image(\Bcal^* \compose P)
        -
        \linspan\paren[\big]{
          \Psd{n}
          \cap
          \set{\Xb}^{\perp}
        }
      }^{\perp}
      \\
      & =
      \Null(\Acal)
      \cap
      \Null(P \compose \Bcal)
      \cap
      \sqbrac[\Big]{
        \linspan\paren[\big]{
          \Psd{n}
          \cap
          \set{\Xb}^{\perp}
        }
      }^{\perp}
      \\
      & =
      \Null(\Acal)
      \cap
      \Null(P \compose \Bcal)
      \cap
      \linspan\setst[\big]{
        \Symmetrize(\Xb \oprod{u}{v})
      }{
        u,v \in \Reals^n
      }.
    \end{split}
  \end{equation*}
  This proves~\eqref{eq:dim-normal-cone-sdp}.

  For the remainder of the proof, suppose that \(\Xb =
  \oprodsym{\xb}\) for some nonzero \(\xb \in \Reals^n\). Note
  that~\eqref{eq:dim-normal-cone-sdp-aux1} specializes to
  \begin{equation}
    \label{eq:conjugate-face-sdp-perp-rank-one}
    \sqbrac[\big]{
      \linspan
      \paren[\big]{
        \Psd{n}
        \cap
        \set{\oprodsym{\xb}}^{\perp}
      }
    }^{\perp}
    =
    \setst*{
      \Symmetrize(\oprod{\xb}{h})
    }{
      h \in \Reals^n
    }
  \end{equation}
  since the RHS of~\eqref{eq:conjugate-face-sdp-perp-rank-one} is a
  linear subspace of \(\Sym{n}\).

  Let \(h \in \Reals^n\). Then
  \(\sqbrac[\big]{\Acal\paren{\Symmetrize(\oprod{\xb}{h})}}_i =
  h^{\transp} A_i \xb\) for \(i \in [p]\) and
  \(\sqbrac[\big]{\Bcal\paren{\Symmetrize(\oprod{\xb}{h})}}_i =
  h^{\transp} B_i \xb\) for \(i \in [q]\). Thus,
  using~\eqref{eq:conjugate-face-sdp-perp-rank-one}, we find that
  \begin{multline*}
    \Null(\Acal)
    \cap
    \Null (P \compose \Bcal)
    \cap
    \sqbrac[\Big]{
      \linspan\paren[\big]{
        \Psd{n}
        \cap
        \set{\Xb}^{\perp}
      }
    }^{\perp}
    \\
    =
    \setst[\Big]{
      \Symmetrize(\oprod{\xb}{h})
    }{
      h \in \paren[\big]{
        \setst{
          A_i \xb
        }{
          i \in [p]
        }
        \cup
        \setst{
          B_i \xb
        }{
          i \in [q] \drop \supp\paren*{\Bcal(\oprodsym{\xb}) - b}
        }
      }^{\perp}
    },
  \end{multline*}
  which has the same dimension as \(\paren[\big]{\setst{A_i\xb}{i \in
      [p]} \cup \setst{B_i\xb}{i \in [q] \drop
      \supp\paren*{\Bcal(\oprodsym{\xb}) - b}}}^{\perp}\) since the
  linear map \(h \in \Reals^n \mapsto \Symmetrize(\oprod{\xb}{h})\) is
  injective. This concludes the proof
  of~\eqref{eq:dim-normal-cone-sdp-rank-one}.
\end{proof}

\section{Vertices of the Lifted Representation of the Theta Body and its
  Variants}

In this section, we shall use Theorem~\ref{thm:dim-normal-cone-sdp} to
characterize the vertices of the lifted theta body, defined
in~\eqref{eq:def-lifted-theta-body}. In fact, we shall identify the
vertices of all the spectrahedra in a slightly more general family,
which includes variations of the lifted theta body that may be used to
define the graph parameters~\(\theta\), \(\theta'\) and~\(\theta^+\),
introduced in~\cite{Lovasz79a, McElieceRR78a, Schrijver79a,
  Szegedy94a}. This will allow us to determine the vertices of some
other spectrahedra that arise as relaxations in combinatorial
optimization; in particular, we shall see the extent to which our
result generalizes the characterization of the vertices of the
elliptope by Laurent and Poljak~\cite{LaurentP95a, LaurentP96a}.

Let~\(V\) be a finite set. Let \(E \subseteq \tbinom{V}{2}\). Recall
the linear maps \(\BoolQuadMap{\setlift{V}}\) and~\(\Acal_E\) defined
on~\eqref{eq:bool-quad-map-def1}--\eqref{eq:edge-entries-def}. We
shall use this notation extensively throughout this section.

We will compute all the vertices of \(\liftedTH(G)\) and its variants,
which we introduce next. For a graph \(G = (V,E)\), define
\begin{gather*}
  \liftedTH{}'(G)
  \coloneqq
  \setst[\big]{
    \Xh \in \Psd{\setlift{V}}
  }{
    \BoolQuadMap{\setlift{V}}(\Xh) = 1 \oplus 0,\,
    \Acal_E(\Xh[V]) = 0,\,
    \Acal_{\overline{E}}(\Xh[V]) \geq 0
  },
  \\
  \TH'(G)
  \coloneqq
  \setst*{
    \diag\paren[\big]{
      \Xh[V]
    }
    \in \Reals^V
  }{
    \Xh \in \liftedTH{}'(G)
  },
  \shortintertext{and}
  \liftedTH{}^+(G)
  \coloneqq
  \setst[\big]{
    \Xh \in \Psd{\setlift{V}}
  }{
    \BoolQuadMap{\setlift{V}}(\Xh) = 1 \oplus 0,\,
    \Acal_E(\Xh[V]) \leq 0,\,
  },
  \\
  \TH^+(G)
  \coloneqq
  \setst*{
    \diag\paren[\big]{
      \Xh[V]
    }
    \in \Reals^V
  }{
    \Xh \in \liftedTH{}^+(G)
  }.
\end{gather*}
It is well known that the Lovász theta number and its variants are the
support functions of these sets, i.e., for a graph \(G = (V,E)\) and
\(w \in \Reals_+^V\), we have
\begin{gather*}
  \theta(G;w)
  =
  \max\setst[\big]{
    \iprod{w}{x}
  }{
    x \in \TH(G)
  },
  \\
  \theta'(G;w)
  =
  \max\setst[\big]{
    \iprod{w}{x}
  }{
    x \in \TH'(G)
  },
  \\
  \theta^+(G;w)
  =
  \max\setst[\big]{
    \iprod{w}{x}
  }{
    x \in \TH^+(G)
  }.
\end{gather*}
We refer the reader to~\cite{Knuth94a, GvozdenovicL08a} and the
references therein for more details.

We may now present our main result, which identifies the vertices
of~\(\liftedTH(G)\), \(\liftedTH{}'(G)\), and~\(\liftedTH{}^+(G)\):
\begin{theorem}
  \label{thm:lifted-TH-vertex-is-rank-one}
  Let \(V\) be a finite set, and let \(E^+, E^- \subseteq
  \tbinom{V}{2}\). Set
  \begin{equation*}
    \ConvexSetLifted
    \coloneqq
    \setst*{
      \Xh \in \Psd{\setlift{V}}
    }{
      \BoolQuadMap{\setlift{V}}(\Xh) = 1 \oplus 0,\,
      \Acal_{E^+}(\Xh[V]) \geq 0,\,
      \Acal_{E^-}(\Xh[V]) \leq 0
    }.
  \end{equation*}
  Let \(\Xh \in \ConvexSetLifted\). Then \(\Xh\) is a vertex
  of~\(\ConvexSetLifted\) if and only if \(\rank(\Xh) = 1\).
\end{theorem}
\begin{proof}
  \newcommand*{\AlmostFullMap}{\Fcal}
  We first prove the `if' part. Let \(\Xh \in \ConvexSetLifted\) be
  rank-one, so that \(\Xh\) is of the form \(\Xh = \oprodsym{(1 \oplus
    \xb)}\) for some \(\xb \in \Reals^V\). Since
  \(\BoolQuadMap{V}(\Xh) = 0\), we have \(\xb \in \set{0,1}^V\). Then
  \(\sqbrac{\BoolQuadMap{\set{0}}^*(e_0)}(1 \oplus \xb) =
  \oprodsym{e_0}(1 \oplus \xb) = e_0\) and, for \(i \in V\), we have
  \(2 \sqbrac{\BoolQuadMap{V}^*(e_i)}(1 \oplus \xb) = 2
  \Symmetrize\paren{\oprod{e_i}{(e_i-e_0)}} (1 \oplus \xb) = (\xb_i -
  1)e_i + \xb_i(e_i - e_0) = \paren[\big]{2\xb_i-1}e_i - \xb_i e_0\).
  These vectors form a basis for~\(\Reals^{\set{0} \cup V}\), whence
  \(\Xh\) is a vertex of~\(\ConvexSetLifted\) by
  Theorem~\ref{thm:dim-normal-cone-sdp}.

  Now we prove that `only if' part. Let \(\Xh\) be a vertex
  of~\(\ConvexSetLifted\). For each \(k \in V\), define
  \begin{equation*}
    \thalf C_k
    \coloneqq
    \Symmetrize\paren{
      \Xh \oprod{e_k}{e_0}
    }
    +
    \sum
    \setst[\bigg]{
      \frac{
        \Xh_{k\ell}
      }{
        \Xh_{\ell\ell}
      }
      \Symmetrize\paren[\big]{
        \Xh \oprodsym{e_{\ell}}
      }
    }{
      \ell \in V,\,
      \Xh_{\ell\ell} > 0
    }.
  \end{equation*}
  For \(E \in \set{E^+, E^-}\), let \(P_E\) denote the orthogonal
  projection onto \(\setst{z \in \Reals^E}{\supp(z) \cap
    \supp\paren{\Acal_E(\Xh[V])} = \emptyset}\). Let \(\AlmostFullMap
  \ffrom \Sym{\setlift{V}} \fto \Reals^V \oplus \Reals^{E+} \oplus
  \Reals^{E-}\) be defined as \(\AlmostFullMap(\Yh) \coloneqq
  \BoolQuadMap{V}(\Yh) \oplus \paren[\big]{P_{E^+} \compose
    \Acal_{E^+}(\Yh[V])} \oplus \paren[\big]{P_{E^-} \compose
    \Acal_{E^-}(\Yh[V])}\) for each \(\Yh \in \Sym{\setlift{V}}\).
  (Note the absence of~\(\set{0}\) in the index set
  of~\(\BoolQuadMap{V}\).) Let us prove that
  \begin{equation}
    \label{eq:lifted-TH-vertex-is-rank-one-aux1}
    C_k \in \Null(\AlmostFullMap).
  \end{equation}

  Let \(i,j \in \set{0} \cup V\). Then
  \begin{equation}
    \label{eq:lifted-TH-vertex-is-rank-one-aux2}
    \begin{split}
      \sqbrac{
        C_k
      }_{ij}
      & =
      \Xh_{ik}
      \Iverson{j = 0}
      +
      \Iverson{i = 0}
      \Xh_{kj}
      +
      \sum
      \setst[\bigg]{
        \frac{
          \Xh_{k\ell}
        }{
          \Xh_{\ell\ell}
        }
        \paren[\Big]{
          \Xh_{i\ell}
          \Iverson{\ell = j}
          +
          \Iverson{\ell = i}
          \Xh_{\ell j}
        }
      }{
        \ell \in V,\,
        \Xh_{\ell\ell} > 0
      }
      \\
      & =
      \Xh_{ik}
      \Iverson{j = 0}
      +
      \Iverson{i = 0}
      \Xh_{kj}
      +
      \sum
      \setst[\bigg]{
        \frac{
          \Xh_{k\ell}
        }{
          \Xh_{\ell\ell}
        }
        \paren[\Big]{
          \Xh_{ij}
          \Iverson{\ell = j}
          +
          \Iverson{\ell = i}
          \Xh_{ij}
        }
      }{
        \ell \in V,\,
        \Xh_{\ell\ell} > 0
      }
      \\
      & =
      \Xh_{ik}
      \Iverson{j = 0}
      +
      \Iverson{i = 0}
      \Xh_{kj}
      +
      \Xh_{ij}
      \sum
      \setst[\bigg]{
        \frac{
          \Xh_{k\ell}
        }{
          \Xh_{\ell\ell}
        }
        \paren[\Big]{
          \Iverson{\ell = j}
          +
          \Iverson{\ell = i}
        }
      }{
        \ell \in V,\,
        \Xh_{\ell\ell} > 0
      }.
    \end{split}
  \end{equation}
  Thus, if \(i,j \in V\) are distinct and \(\Xh_{ij} = 0\), then
  \(\sqbrac{C_k}_{ij} = 0\). Let \(i \in V\). Then
  \begin{equation*}
    \sqbrac{
      C_k
    }_{ii}
    =
    \Xh_{ii}
    \sum
    \setst[\bigg]{
      \frac{
        \Xh_{k\ell}
      }{
        \Xh_{\ell\ell}
      }
      2\Iverson{\ell = i}
    }{
      \ell \in V,\,
      \Xh_{\ell\ell} > 0
    }
    =
    2\Iverson{\Xh_{ii} > 0}
    \Xh_{ki}
    =
    2\Xh_{ki}
  \end{equation*}
  whereas
  \begin{equation}
    \label{eq:lifted-TH-vertex-is-rank-one-aux3}
    \sqbrac{
      C_k
    }_{i0}
    =
    \Xh_{ik}
    +
    \Xh_{i0}
    \sum
    \setst[\bigg]{
      \frac{
        \Xh_{k\ell}
      }{
        \Xh_{\ell\ell}
      }
      \Iverson{\ell = i}
    }{
      \ell \in V,\,
      \Xh_{\ell\ell} > 0
    }
    =
    \Xh_{ik}
    +
    \Iverson{\Xh_{ii} > 0}
    \Xh_{ki}
    =
    2\Xh_{ki}.
  \end{equation}
  This concludes the proof
  of~\eqref{eq:lifted-TH-vertex-is-rank-one-aux1}.

  We claim that
  \begin{equation}
    \label{eq:lifted-TH-vertex-is-rank-one-aux4}
    \text{if }
    k, \ell \in V
    \text{ are such that }
    \Xh_{kk} > 0
    \text{ and }
    \Xh_{\ell\ell} > 0,
    \text{then }
    \Xh_{kk} = \Xh_{\ell\ell} = \Xh_{k\ell}.
  \end{equation}
  Let \(k, \ell \in V\) such that \(\Xh_{kk} > 0\) and
  \(\Xh_{\ell\ell} > 0\). Set
  \begin{equation*}
    D
    \coloneqq
    \frac{1}{\Xh_{kk}}
    C_k
    -
    \frac{1}{\Xh_{\ell\ell}}
    C_{\ell}.
  \end{equation*}
  Note that \(\sqbrac{C_k}_{00} = 2\Xh_{0k} = 2\Xh_{kk}\) and
  \(\sqbrac{C_{\ell}}_{00} = 2\Xh_{0\ell} = 2\Xh_{\ell\ell}\), whence
  \(D_{00} = 0\). Hence, \(D \in \Null(\BoolQuadMap{\set{0}})\).
  By~\eqref{eq:lifted-TH-vertex-is-rank-one-aux1}, we also have \(D
  \in \Null(\AlmostFullMap)\). Thus, by
  Theorem~\ref{thm:dim-normal-cone-sdp}, we must have \(D = 0\). Now
  from~\eqref{eq:lifted-TH-vertex-is-rank-one-aux3} we get
  \begin{equation*}
    0
    =
    D_{k0}
    =
    \frac{
      \sqbrac{
        C_k
      }_{k0}
    }{
      \Xh_{kk}
    }
    -
    \frac{
      \sqbrac{
        C_{\ell}
      }_{k0}
    }{
      \Xh_{\ell\ell}
    }
    =
    \frac{
      2 \Xh_{kk}
    }{
      \Xh_{kk}
    }
    -
    \frac{
      2 \Xh_{\ell k}
    }{
      \Xh_{\ell\ell}
    }
    \implies
    \Xh_{\ell\ell}
    =
    \Xh_{k \ell}
  \end{equation*}
  and
  \begin{equation*}
    0
    =
    D_{\ell 0}
    =
    \frac{
      \sqbrac{
        C_k
      }_{\ell 0}
    }{
      \Xh_{kk}
    }
    -
    \frac{
      \sqbrac{
        C_{\ell}
      }_{\ell 0}
    }{
      \Xh_{\ell\ell}
    }
    =
    \frac{
      2 \Xh_{k \ell}
    }{
      \Xh_{kk}
    }
    -
    \frac{
      2 \Xh_{\ell\ell}
    }{
      \Xh_{\ell\ell}
    }
    \implies
    \Xh_{kk}
    =
    \Xh_{k \ell}.
  \end{equation*}
  This concludes the proof
  of~\eqref{eq:lifted-TH-vertex-is-rank-one-aux4}.

  From~\eqref{eq:lifted-TH-vertex-is-rank-one-aux4} we find that there
  exists \(\eta \in \Reals\) such that
  \begin{equation}
    \label{eq:lifted-TH-vertex-is-rank-one-aux5}
    \Xh
    =
    (1-\eta)
    \sqbrac[\big]{
      \oprodsym{
        \paren{
          1 \oplus 0
        }
      }
    }
    +
    \eta
    \sqbrac[\big]{
      \oprodsym{
        \paren{
          1 \oplus \incidvector{S}
        }
      }
    }
  \end{equation}
  where \(S \coloneqq \supp(\diag(\Xh[V]))\). If \(S = \emptyset\),
  the proof is complete, so assume that \(S \neq \emptyset\). Then
  \(\Xh \succeq 0\) is equivalent to~\(\eta \in [0,1]\). If \(\eta =
  0\) the proof is complete, so assume \(\eta > 0\). Then
  \eqref{eq:lifted-TH-vertex-is-rank-one-aux5} describes the extreme
  point~\(\Xh\) as a convex combination of two distinct points
  of~\(\ConvexSetLifted\), from which we conclude that \(\eta \in
  \set{0,1}\). Now \(\rank(\Xh) = 1\) follows
  from~\eqref{eq:lifted-TH-vertex-is-rank-one-aux5}.
\end{proof}

We immediately obtain from
Theorem~\ref{thm:lifted-TH-vertex-is-rank-one} the vertices of all the
lifted theta bodies defined above:
\begin{corollary}
  \label{cor:liftedTH-vertices}
  Let \(G = (V,E)\) be a graph. Let \(\ConvexSetLifted \in
  \set[\big]{\liftedTH(G), \liftedTH{}'(G), \liftedTH{}^+(G)}\). Then
  a point \(\Xh\) of~\(\ConvexSetLifted\) is a vertex
  of~\(\ConvexSetLifted\) if and only if \(\rank(\Xh) = 1\). Thus, the
  vertices of~\(\ConvexSetLifted\) are precisely the matrices of the
  form \(\oprodsym{(1 \oplus \incidvector{S})}\) where \(S \subseteq
  V\) is a stable set of~\(G\).
\end{corollary}
\begin{proof}
  Immediate from Theorem~\ref{thm:lifted-TH-vertex-is-rank-one}: for
  \(\ConvexSetLifted = \liftedTH(G)\), take \(E^+ \coloneqq E^-
  \coloneqq E\); for \(\ConvexSetLifted = \liftedTH{}'(G)\), take
  \(E^+ \coloneqq \tbinom{V}{2}\) and \(E^- \coloneqq E\); for
  \(\ConvexSetLifted = \liftedTH{}^+(G)\), take \(E^+ \coloneqq
  \emptyset\) and \(E^- \coloneqq E\).
\end{proof}

Let \(V\) be a finite set. Define
\begin{subequations}
  \begin{gather}
    \label{eq:BQbody-def}
    \BQbody{\setlift{V}}
    \coloneqq
    \setst[\big]{
      \Xh \in \Psd{\setlift{V}}
    }{
      \BoolQuadMap{\setlift{V}}(\Xh) = 1 \oplus 0
    },
    \\
    \BQbody{\setlift{V}}'
    \coloneqq
    \setst[\big]{
      \Xh \in \BQbody{\setlift{V}}
    }{
      \Xh[V] \geq 0
    },
    \\
    \BQbody{\setlift{V}}''
    \coloneqq
    \setst[\big]{
      \Xh \in \BQbody{\setlift{V}}
    }{
      \iprod[\big]{
        \Symmetrize\paren[\big]{
          \oprod{
            (e_0-e_i)
          }{
            (e_0-e_j)
          }
        }
      }{
        \Xh
      }
      \geq
      0,\,
      \forall ij \in \tbinom{V}{2}
    }.
  \end{gather}
\end{subequations}
These sets are well-known relaxations for the \emph{boolean quadric
  polytope} \(\conv\setst[\big]{\oprodsym{(1 \oplus x)}}{x \in
  \set{0,1}^V}\). Also, set
\begin{equation*}
  \FlipMatrix
  \coloneqq
  \oprodsym{e_0}
  +
  \sum_{i \in V} \oprod{e_i}{(e_0 - e_i)}
  \in
  \Reals^{
    \paren{\setlift{V}}
    \times
    \paren{\setlift{V}}
  }.
\end{equation*}
Note that \(\FlipMatrix (1 \oplus \incidvector{S}) = (1 \oplus
\incidvector{V \drop S})\) for each \(S \subseteq V\). In fact,
\(\FlipMatrix\) is its own inverse. It is easy to check that
\(\FlipMap\) is an automorphism of~\(\BQbody{\setlift{V}}\), and that
\begin{equation}
  \label{eq:flip-BQbody-isomorphism}
  \FlipMap\paren[\big]{
    \BQbody{\setlift{V}}'
  }
  =
  \BQbody{\setlift{V}}''.
\end{equation}

\begin{corollary}
  \label{cor:BQbody-vertices}
  Let \(V\) be a finite set. Let \(\ConvexSetLifted \in
  \set[\big]{\BQbody{\setlift{V}}, \BQbody{\setlift{V}}',
    \BQbody{\setlift{V}}''}\). Then a point \(\Xh\)
  of~\(\ConvexSetLifted\) is a vertex of~\(\ConvexSetLifted\) if and
  only if \(\rank(\Xh) = 1\). Thus, the vertices
  of~\(\ConvexSetLifted\) are precisely the matrices of the form
  \(\oprodsym{(1 \oplus \incidvector{S})}\) where \(S \subseteq V\).
\end{corollary}
\begin{proof}
  For \(\ConvexSetLifted \in \set[\big]{\BQbody{\setlift{V}},
    \BQbody{\setlift{V}}'}\), this follows from
  Corollary~\ref{cor:liftedTH-vertices} via
  Lemma~\ref{lemma:invariance}, since \(\BQbody{\setlift{V}} =
  \liftedTH\paren{\overline{K_V}}\) and \(\BQbody{\setlift{V}}'=
  \liftedTH{}'\paren{\overline{K_V}}\), where \(K_V\) denotes the
  complete graph on~\(V\). For \(\ConvexSetLifted =
  \BQbody{\setlift{V}}''\), this follows from the previous sentence
  together with~\eqref{eq:flip-BQbody-isomorphism} and
  Lemma~\ref{lemma:invariance}.
\end{proof}

Let \(V\) be a finite set. Define
\begin{gather*}
  \Elliptope{\setlift{V}}'
  \coloneqq
  \setst[\big]{
    \Xh \in \Elliptope{\setlift{V}}
  }{
    \iprod[\big]{
      \Symmetrize\paren[\big]{
        \oprod{
          (e_0+e_i)
        }{
          (e_0+e_j)
        }
      }
    }{
      \Xh
    }
    \geq
    0,\,
    \forall ij \in \tbinom{V}{2}
  },
  \\
  \Elliptope{\setlift{V}}''
  \coloneqq
  \setst[\big]{
    \Xh \in \Elliptope{\setlift{V}}
  }{
    \iprod[\big]{
      \Symmetrize\paren[\big]{
        \oprod{
          (e_0-e_i)
        }{
          (e_0-e_j)
        }
      }
    }{
      \Xh
    }
    \geq
    0,\,
    \forall ij \in \tbinom{V}{2}
  }.
\end{gather*}
Like~\(\Elliptope{\setlift{V}}\), these sets are also relaxations for
the \(\conv\setst[\big]{\oprodsym{(1 \oplus x)}}{x \in
  \set{\pm1}^V}\), which is a variant of the cut polytope. Also, set
\begin{equation*}
  \SignToIncidMatrix
  \coloneqq
  \thalf
  \sum_{
    \mathclap{i \in \setlift{V}}
  }
  \oprod{
    e_i
  }{
    (e_0 + e_i)
  }
  \in
  \Reals^{
    \paren{\setlift{V}}
    \times
    \paren{\setlift{V}}
  }.
\end{equation*}
Note that \(\SignToIncidMatrix \paren[\big]{1 \oplus (\incidvector{S}
  - \incidvector{V \drop S})} = 1 \oplus \incidvector{S}\) for each
\(S \subseteq V\). It is easy to check that \(\SignToIncidMatrix\) is
invertible and
\begin{subequations}
  \label{eq:elliptopes-from-BQbodies}
  \begin{gather}
    \SignToIncidMap(\Elliptope{\setlift{V}})
    =
    \BQbody{\setlift{V}},
    \\
    \SignToIncidMap(\Elliptope{\setlift{V}}')
    =
    \BQbody{\setlift{V}}',
    \\
    \SignToIncidMap(\Elliptope{\setlift{V}}'')
    =
    \BQbody{\setlift{V}}''.
  \end{gather}
\end{subequations}
The linear isomorphism~\(\SignToIncidMap\) is quite interesting in the
sense that it also maps the cut polytope to the boolean quadric
polytope, the sets for which~\(\Elliptope{\setlift{V}}\)
and~\(\BQbody{\setlift{V}}\) are relaxations, respectively;
see~\cite{DeSimone90a, LaurentPR97a}.
\begin{corollary}
  \label{cor:elliptope-variants-vertices}
  Let \(V\) be a finite set. Let \(\ConvexSetLifted \in
  \set[\big]{\Elliptope{\setlift{V}}, \Elliptope{\setlift{V}}',
    \Elliptope{\setlift{V}}''}\). Then a point \(\Xh\) of
  \(\,\ConvexSetLifted\) is a vertex of~\(\,\ConvexSetLifted\) if and
  only if \(\rank(\Xh) = 1\). Thus, the vertices
  of~\(\,\ConvexSetLifted\) are precisely the matrices of the form
  \(\oprodsym{(1 \oplus x_S)}\) where \(x_S = \incidvector{S} -
  \incidvector{V \drop S}\) for some \(S \subseteq V\).
\end{corollary}
\begin{proof}
  Immediate from Corollary~\ref{cor:BQbody-vertices}
  and~\eqref{eq:elliptopes-from-BQbodies} via
  Lemma~\ref{lemma:invariance}.
\end{proof}

Corollary~\ref{cor:elliptope-variants-vertices} allows us to gauge the
extent to which Corollary~\ref{cor:liftedTH-vertices} generalizes
Corollary~\ref{cor:elliptope-vertices}: the latter result
characterizes the vertices for one convex set for each positive
integer~\(n\), whereas the former does the same for all positive
integers~\(n\) and all graphs with~\(n\) nodes.

Kleinberg and Goemans~\cite{KleinbergG98a} presented SDP relaxations
for the vertex cover problem. For a graph \(G = (V,E)\), the feasible
regions of their relaxations are:
\begin{gather*}
  \liftedKleinbergGoemansVC(G)
  \coloneqq
  \setst[\big]{
    \Xh \in \Elliptope{\setlift{V}}
  }{
    \iprod[\big]{
      \Symmetrize\paren[\big]{
        \oprod{
          (e_0-e_i)
        }{
          (e_0-e_j)
        }
      }
    }{
      \Xh
    }
    =
    0,\,
    \forall ij \in E
  },
  \\
  \liftedKleinbergGoemansVC{}'(G)
  \coloneqq
  \liftedKleinbergGoemansVC(G) \cap \Elliptope{\setlift{V}}''.
\end{gather*}
\begin{corollary}
  \label{cor:VC-vertices}
  Let \(G = (V,E)\) be a graph. Let \(\ConvexSetLifted \in
  \set[\big]{\liftedKleinbergGoemansVC(G),
    \liftedKleinbergGoemansVC{}'(G)}\). Then a point \(\Xh\)
  of~\(\ConvexSetLifted\) is a vertex of~\(\ConvexSetLifted\) if and
  only if \(\rank(\Xh) = 1\). Thus, the vertices
  of~\(\ConvexSetLifted\) are precisely the matrices of the form
  \(\oprodsym{(1 \oplus \incidvector{S})}\) where \(S \subseteq V\) is
  a vertex cover of~\(G\).
\end{corollary}
\begin{proof}
  Immediate from Corollary~\ref{cor:liftedTH-vertices} via
  Lemma~\ref{lemma:invariance}, since we have
  \begin{gather*}
    \liftedKleinbergGoemansVC(G)
    =
    \paren*{
      \SignToIncidMap^{-1}
      \compose
      \FlipMap
    }
    \paren[\big]{
      \,
      \liftedTH(G)
    },
    \\
    \liftedKleinbergGoemansVC{}'(G)
    =
    \paren*{
      \SignToIncidMap^{-1}
      \compose
      \FlipMap
    }
    \paren[\big]{
      \,
      \liftedTH{}'(G)
    }.\qedhere
  \end{gather*}
\end{proof}

The Lovász theta number is sometimes presented using another SDP, in a
smaller dimensional space. We shall now show that the vertices of the
feasible region of this SDP do not coincide with what would be
considered its exact solutions:
\begin{theorem}
  \label{thm:theta3-vertices}
  Let \(V\) be a finite set, and let \(E^+, E^- \subseteq
  \tbinom{V}{2}\). Set
  \begin{equation*}
    \ConvexSet
    \coloneqq
    \setst*{
      X \in \Psd{V}
    }{
      \trace(X) = 1,\,
      \Acal_{E^+}(X) \geq 0,\,
      \Acal_{E^-}(X) \leq 0
    },
  \end{equation*}
  \(H \coloneqq (V,E^{+1} \cup E^{-1})\), and \(n \coloneqq
  \card{V}\). Then the set of vertices of~\(\ConvexSet\) is
  \(\setst[\big]{\oprodsym{e_k}}{\deg_H(k) = n-1}\).
\end{theorem}
\begin{proof}
  We first show that
  \begin{equation}
    \label{eq:theta3-vertices-aux1}
    \text{
      if \(\Xb\) is a vertex of~\(\ConvexSet\), then \(\Xb =
      \oprodsym{e_k}\) for some \(k \in V\).
    }
  \end{equation}
  Let \(\Xb\) be a vertex of~\(\ConvexSet\). Let \(k, \ell \in V\) be
  distinct. Set
  \begin{equation*}
    \thalf D
    \coloneqq
    \Xb_{\ell\ell}
    \Symmetrize\paren{
      \Xb \oprodsym{e_k}
    }
    -
    \Xb_{kk}
    \Symmetrize\paren{
      \Xb \oprodsym{e_{\ell}}
    }.
  \end{equation*}
  If \(i,j \in V\), then
  \begin{equation*}
    \begin{split}
      D_{ij}
      & =
      \Xb_{\ell\ell}
      \paren[\Big]{
        \Xb_{ik}
        \Iverson{k = j}
        +
        \Iverson{k = i}
        \Xb_{kj}
      }
      -
      \Xb_{kk}
      \paren[\Big]{
        \Xb_{i\ell}
        \Iverson{\ell = j}
        +
        \Iverson{\ell = i}
        \Xb_{\ell j}
      }
      \\
      & =
      \Xb_{\ell\ell}
      \Xb_{ij}
      \paren[\big]{
        \Iverson{k = j}
        +
        \Iverson{k = i}
      }
      -
      \Xb_{kk}
      \Xb_{ij}
      \paren[\big]{
        \Iverson{\ell = j}
        +
        \Iverson{\ell = i}
      }
      \\
      & =
      \Xb_{ij}
      \sqbrac[\Big]{
        \Xb_{\ell\ell}
        \paren[\big]{
          \Iverson{k = j}
          +
          \Iverson{k = i}
        }
        -
        \Xb_{kk}
        \paren[\big]{
          \Iverson{\ell = j}
          +
          \Iverson{\ell = i}
        }
      }.
    \end{split}
  \end{equation*}
  For \(ij \in \tbinom{V}{2}\), we clearly have \(D_{ij} = 0\)
  whenever \(\Xb_{ij} = 0\). We also have
  \begin{equation*}
    \trace(D)
    =
    D_{kk}
    +
    D_{\ell\ell}
    =
    \paren[\big]{
      2 \Xb_{kk} \Xb_{\ell\ell}
    }
    +
    \paren[\big]{
      - 2 \Xb_{\ell\ell} \Xb_{kk}
    }
    =
    0.
  \end{equation*}
  Note that \(\card{V}^{-1} I\) lies in \(\ConvexSet \cap \Pd{V}\), so
  we may apply Theorem~\ref{thm:dim-normal-cone-sdp} to get \(D = 0\).
  Thus, \(0 = D_{kk} = 2 \Xb_{kk} \Xb_{\ell\ell}\). Since~\(k\)
  and~\(\ell\) were arbitrary, \eqref{eq:theta3-vertices-aux1} follows
  from~\(\trace(\Xb) = 1\).

  We will now show that,
  \begin{equation}
    \label{eq:theta3-vertices-aux2}
    \text{
      for \(k \in V\), the point \(\oprodsym{e_k}\)
      is a vertex of~\(\ConvexSet\) if
      and only if \(\deg_H(k) = n-1\).
    }
  \end{equation}
  Let \(k \in V\). Set \(E \coloneqq E^+ \cup E^-\). By
  Theorem~\ref{thm:dim-normal-cone-sdp}, the point \(\oprodsym{e_k}\)
  is a vertex of~\(\ConvexSet\) if and only if \(\set{e_k} \cup
  \setst{\Symmetrize(\oprod{e_i}{e_j}) e_k}{ij \in E}\)
  spans~\(\Reals^V\). The latter set is \(\set{e_k} \cup
  \setst{\Iverson{j=k} e_i + \Iverson{i=k} e_j}{ij \in E} = \set{e_k}
  \cup \setst{e_j}{jk \in E}\), so it spans~\(\Reals^V\) precisely
  when \(\deg_H(k) = n-1\).

  The result now follows from~\eqref{eq:theta3-vertices-aux1}
  and~\eqref{eq:theta3-vertices-aux2}.
\end{proof}

\begin{corollary}
  Let \(G = (V,E)\) be a graph. Set \(P \coloneqq \setst{k \in
    V}{\deg_G(k) = \card{V} - 1}\). Then
  \begin{rmlist*}
  \item the set of vertices of \(\setst{X \in \Psd{V}}{\trace(X) =
      1,\, \Acal_E(X) = 0}\) is \(\setst{\oprodsym{e_k}}{k \in P}\);
  \item the set of vertices of \(\setst{X \in \Psd{V}}{\trace(X) =
      1,\, \Acal_E(X) = 0,\, \Acal_{\overline{E}}(X) \geq 0}\) is
    \(\setst{\oprodsym{e_k}}{k \in V}\);
  \item the set of vertices of \(\setst{X \in \Psd{V}}{\trace(X) =
      1,\, \Acal_E(X) \leq 0}\) is \(\setst{\oprodsym{e_k}}{k \in
      P}\).
  \end{rmlist*}
\end{corollary}
\begin{proof}
  Immediate from Theorem~\ref{thm:theta3-vertices}, as in the proof of
  Corollary~\ref{cor:liftedTH-vertices}.
\end{proof}

The results in this section significantly extend the
combinatorially-inspired spectrahedra whose vertices are completely
understood. However, we do not know the set of vertices of some of
their simplest variants, such as \(\BQbody{\setlift{V}}' \cap
\BQbody{\setlift{V}}''\) or even
\begin{equation*}
  \setst[\big]{
    \Xh \in \BQbody{\setlift{V}}
  }{
    \iprod[\big]{
      \Symmetrize\paren[\big]{
        \oprod{e_i}{(e_i - e_j)}
      }
    }{
      \Xh
    }
    \geq 0,\,
    \forall (i,j) \in V \times V
  };
\end{equation*}
the constraints of the latter usually appear in spectrahedra arising
from the lift-and-project operator of Lovász and
Schrijver~\cite{LovaszS91a}. This is just a hint of the complexity of
the vertex structure of spectrahedra that we warned about in the
introduction. We roughly discuss some other difficulties next.

When considering sufficient conditions which bound the rank of
vertices of a spectrahedron, such as the ones from
Theorem~\ref{thm:normal-cone-modular-rank} and
Proposition~\ref{prop:free-offdiag-rank-bound}, ideally one seeks to
obtain coordinate-free conditions that are easy to check and that have
a built-in detection for a change of basis. Let us use
Theorem~\ref{thm:normal-cone-modular-rank} to explain this. Suppose we
replace the rank hypothesis from that theorem with the condition that
\(A_i A_j = 0\) for distinct \(i,j \in [n]\). Note that we eventually
reach this assumption in~\eqref{eq:dim-normal-cone-from-Null-aux2} in
its proof. Then the modified theorem would be applicable to the
elliptope~\(\Elliptope{\setlift{V}}\), but not to its linear
isomorphic image
\begin{equation*}
  \setst*{
    \Xh \in \Psd{\setlift{V}}
  }{
    \iprod[\big]{\oprodsym{(e_0-2e_i)}}{\Xh} = 1\,
    \forall i \in \setlift{V}
  },
\end{equation*}
which is nothing but \(\BQbody{\setlift{V}}\). What happened in this
case was that we have the following equivalence: there exists a
nonsingular \(L \in \Reals^{n \times n}\) such that \(\CongMap{L}(A_i)
\CongMap{L}(A_j) = 0\) for distinct \(i,j \in [n]\) if and only if the
rank condition from Theorem~\ref{thm:normal-cone-modular-rank} holds.
That is, a simple algebraic condition subsumes an existential
predicate about a convenient basis; the rank condition factors out the
trivial congruences. This is in contrast with the existential
hypothesis from Proposition~\ref{prop:free-offdiag-rank-bound}, which
is harder to check, and thus harder to apply. However,
Theorem~\ref{thm:normal-cone-modular-rank} is not yet entirely
coordinate-free; this may be seen from the fact that it does not apply
directly to~\(\BQbody{\setlift{V}}\) using its description
in~\eqref{eq:BQbody-def}, since the theorem requires the RHS of the
defining linear equations to be nonzero everywhere. In this sense,
Theorem~\ref{thm:normal-cone-modular-rank} still has some room for
improvement.

The algebraic aspects just described have a complementary role to
geometry in some situations. For instance, it is easy to see how to
start with a spectrahedron all of whose vertices have rank one and
transform it into one that has all vertices of rank two; one could
take a direct sum with a constant nonzero block, and apply a
congruence transformation to ``hide'' the triviality of this
transformation. Here the geometric aspect of the transformation is
trivial. However, a broad sufficient condition to bound the rank of
vertices needs to factor out all these congruences. This seems hard to
describe algebraically without an existential hypothesis. On the other
direction, Corollary~\ref{cor:smash-to-rank-one} describes a
transformation of spectrahedra that is trivial in terms of algebra,
but geometrically it modifies the boundary structure drastically.

The above results indicate that the approach presented here and in the
previous literature we cited, may lead
to further fruitful results.  To indicate some of this potential, we
move to some other aspects of the boundary structure beyond the
vertices, but continue to utilize the characterizations of normal cone
and duality themes.

\section{Strict Complementarity}
\label{sec:strict-compl}

We continue considering the consequences of the dual
characterization~\eqref{eq:conic-lp-normal-cone} of the normal cone at
a boundary point which has been very fruitful so far. Note that
Proposition~\ref{prop:conic-lp-normal-cone}, which is a
dual characterization of the normal cone at \(\xb\), shows explicitly
that the normal cone at \(\xb\) of the feasible region of a conic
optimization problem is the Minkowski sum of a polyhedral cone
(defined by~\(\Acal\), \(\Bcal\), and \(\xb\)) and the conjugate of
the minimal face of \(\Cone \subseteq \Euclidean\) containing \(\xb\).
By taking the relative interior of both sides of this
characterization, we shall find a strong connection to strict
complementarity. We first recall a definition by
Pataki~\cite{Pataki96a}.
\begin{definition}
  \label{def:conic-lp-setting}
  Let \(\Cone \subseteq \EuclideanA\) be a pointed closed convex cone
  with nonempty interior. Let \(\Acal \ffrom \EuclideanA \fto
  \EuclideanB^*\) be a linear function. Let \(c \in \EuclideanA^*\)
  and \(b \in \EuclideanB^*\). Set
  \begin{subequations}
    \label{eq:conic-primal-dual-pair}
    \begin{gather}
      \ConvexSet_P
      \coloneqq
      \setst{
        x \in \Cone
      }{
        \Acal(x) = b
      },
      \\
      \ConvexSet_D
      \coloneqq
      \setst{
        s \in \Cone^*
      }{
        s \in \Image(\Acal^*) - c
      }.
    \end{gather}
  \end{subequations}
  We say that a pair \((\xb,\bar{s}) \in \ConvexSet_P \times
  \ConvexSet_D\) is \emph{strictly complementary} if there exists a
  face~\(F\) of~\(\Cone\) such that
  \begin{equation*}
    \xb \in \relint(F)
    \text{ and }
    \bar{s} \in \relint(F^{\triangle}).
  \end{equation*}
\end{definition}
In the above \(F^{\triangle} \coloneqq \Cone^* \cap F^{\perp}\) is the
conjugate face of~\(F\). Note that the condition \(\xb \in
\relint(F)\) for a face~\(F\) of~\(\Cone\) is equivalent to the fact
that~\(F\) is the smallest face of~\(\Cone\) that contains~\(\xb\);
see, e.g., \cite[Theorem~18.1]{Rockafellar97a}. Thus, if we
define~\(F\) as the smallest face of~\(\Cone\) containing~\(\xb\),
then \(F^{\triangle} = \Cone^* \cap \set{\xb}^{\perp}\). With this in
mind, the next observation becomes clear from
Proposition~\ref{prop:conic-lp-normal-cone} (it appears in a slightly
different form in~\cite[Sec.~2]{BolteDL11a}):
\begin{proposition}
  \label{prop:4.1}
  In the context of Definition~\ref{def:conic-lp-setting}, suppose
  that \(\ConvexSet_P \cap \interior(\Cone) \neq \emptyset\). Let
  \(\xb \in \ConvexSet_P\). Then there exists \(\bar{s} \in
  \ConvexSet_D\) such that \((\xb,\bar{s})\) is strictly complementary
  if and only if \(c \in
  \relint\paren{\NormalCone{\ConvexSet_P}{\xb}}\).
\end{proposition}
\begin{proof}
  The condition for strict complementarity of a pair \((\xb,\bar{s})\)
  requires a face~\(F\) of~\(\Cone\) to satisfy \(\xb \in
  \relint(F)\). Since our primal feasible~\(\xb\) is fixed, the
  face~\(F\) is also fixed to be the minimal face of~\(\Cone\)
  containing~\(\xb\). Thus, given \(\bar{s} \in \ConvexSet_D\), strict
  complementarity of \((\xb,\bar{s})\) is equivalent to the membership
  \(\bar{s} \in \relint(F^{\triangle}) = \relint\paren{\Cone^* \cap
    \set{\xb}^{\perp}}\). Under the assumption that \(\ConvexSet_P
  \cap \interior{\Cone} \neq \emptyset\), we have by
  Proposition~\ref{prop:conic-lp-normal-cone} that
  \begin{equation}
    \label{eq:normal-strict-compl-aux1}
    \relint\paren[\big]{
      \NormalCone{\ConvexSet_P}{\xb}
    }
    =
    \Image(\Acal^*)
    -
    \relint\paren[\big]{
      \Cone^* \cap \set{\xb}^{\perp}
    }.
  \end{equation}

  Suppose that \(\bar{s} \in \ConvexSet_D\) is such that \(\bar{s} \in
  \relint\paren{\Cone^* \cap \set{\xb}^{\perp}}\). Since \(\bar{s} \in
  \ConvexSet_D\), we have \(c \in \Image(\Acal^*) - \bar{s} \subseteq
  \Image(\Acal^*) - \relint\paren{\Cone^* \cap \set{\xb}^{\perp}} =
  \relint\paren[\big]{\NormalCone{\ConvexSet_P}{\xb}}\)
  by~\eqref{eq:normal-strict-compl-aux1}. For the converse, suppose
  that \(c \in \relint\paren[\big]{\NormalCone{\ConvexSet_P}{\xb}}\).
  Then by~\eqref{eq:normal-strict-compl-aux1} there exists \(\bar{s}
  \in \relint\paren[\big]{\Cone^* \cap \set{\xb}^{\perp}}\) such that
  \(c \in \Image(\Acal^*) - \bar{s}\). In particular, \(\bar{s} \in
  \ConvexSet_D\). Thus, \((\xb,\bar{s})\) is strictly complementary.
\end{proof}

The above proposition already implies that strict complementarity is
\emph{locally generic} in the following sense:
\begin{corollary}
  In the context of Definition~\ref{def:conic-lp-setting}, suppose
  that \(\ConvexSet_P \cap \interior(\Cone) \neq \emptyset\). Let
  \(\xb \in \ConvexSet_P\). Consider the set
  \(\NormalCone{\ConvexSet_P}{\xb}\) of all \(c \in \Euclidean^*\) for
  which~\(\xb\) is optimal for \(\sup\setst{\iprod{c}{x}}{x \in
    \ConvexSet_P}\). Set \(d \coloneqq
  \dim\paren{\NormalCone{\ConvexSet_P}{\xb}}\). Then, the set
  \begin{equation*}
    \setst[\big]{
      c \in \NormalCone{\ConvexSet_P}{\xb}
    }{
      \text{%
        there does not exist \(\bar{s} \in \ConvexSet_D\) such that
        \((\xb,\bar{s})\) is strictly complementary%
      }
    }
  \end{equation*}
  is of measure zero with respect to \(d\)\nbd-dimensional Hausdorff
  measure.
\end{corollary}
\begin{proof}
  Normal cone is closed and convex, and we may assume that it is
  pointed. Then, we can analyze its boundary structure by taking a
  cross-section of it via intersection by a hyperplane whose normal is
  defined by an interior point of the cone dual to the normal cone (in
  the \(d\)\nbd-dimensional affine span of the cone). Let
  \(\ConvexSet\) denote this cross-section (whose dimension is
  \(d-1\)). It is well-known that the set of all
  \((d-2)\)\nbd-dimensional faces of such a convex set \(\ConvexSet\)
  is a countable set and furthermore the union of the relative
  boundaries of these faces have zero \((d-2)\)\nbd-dimensional
  Hausdorff measure (a result of Larman~\cite{Larman71a}). Therefore,
  for the \((d-1)\)\nbd-dimensional convex set \(\ConvexSet\), its
  boundary has a zero \((d-1)\)\nbd-dimensional Hausdorff measure.
  Hence, the relative boundary of the normal cone in consideration is
  of measure zero with respect to \(d\)\nbd-dimensional Hausdorff
  measure. Therefore, the claim follows from
  Proposition~\ref{prop:4.1}.
\end{proof}

There are similar strict complementarity results in the literature
starting with Alizadeh, Haeberly and Overton~\cite{AlizadehHO97a},
Pataki and the second author~\cite{PatakiT01a}, Gortler and
Thurston~\cite{GortlerT10a}, Nie, Ranestad and
Sturmfels~\cite{NieRS10a}, and Drusvyatskiy and
Lewis~\cite{DrusvyatskiyL11a}. All of these results are
\emph{generic}.  Many of these papers also address various
related notions of nondegeneracy and establish that it too is generic.
However, it is well-known in LP literature that degeneracy arises often
in applications and in many cases ``naturally.''  Therefore, it is
of interest to characterize when a certain geometric/algebraic condition
can guarantee nondegeneracy or strict complementarity.

We shall next present a characterization of strict complementarity which may be
helpful in proving that some specific SDPs satisfy strict
complementarity.  First we recall an elementary result in convex analysis
(for the sake of completeness
a proof is included in the appendix):
\begin{proposition}
  \label{prop:compact-gauge-closed}
  Let \(\ConvexSet \subseteq \Euclidean\) be a compact convex set.
  Then the gauge function \(\gauge{\ConvexSet}{\cdot}\)
  of~\(\ConvexSet\) is closed.
\end{proposition}

Using Proposition~\ref{prop:compact-gauge-closed}, we slightly extend
a characterization of the exposed faces of the polar
from~\cite{BolteDL11a} (again, for the sake of completeness, a proof
is included in the appendix):
\begin{theorem}
  \label{thm:normal-cone-exposed-faces}
  Let \(\ConvexSet \subseteq \Euclidean\) be a compact convex set.
  Then the nonempty exposed faces of~\(\polar{\ConvexSet}\) other
  than~\(\polar{\ConvexSet}\) itself are precisely the nonempty sets
  of the form
  \begin{equation}
    \label{eq:normal-cone-exposed-faces}
    \Face_{\xb}
    \coloneqq
    \setst{
      y \in \NormalCone{\ConvexSet}{\xb}
    }{
      \iprod{y}{\xb} = 1
    }
  \end{equation}
  as \(\xb\) ranges over~\(\ConvexSet\). Moreover, for any such face,
  \begin{equation}
    \label{eq:polar-face-relint}
    \relint\paren*{\Face_{\xb}}
    =
    \setst*{
      y \in \relint\paren*{\NormalCone{\ConvexSet}{\xb}}
    }{
      \iprod{y}{\xb} = 1
    }.
  \end{equation}
\end{theorem}

Now we can characterize \emph{exactly} the existence of strictly
complementary solutions for linear programs in conic form for a rich
class of objective functions.

\begin{theorem}
  \label{thm:strict-compl-exposed-faces}
  Let \(\Cone \subseteq \Euclidean\) be a pointed closed convex cone
  with nonempty interior. Let \(\Acal \ffrom \Euclidean \fto
  \EuclideanB^*\) be a linear function, and let \(b \in
  \EuclideanB^*\). Set \(\ConvexSet \coloneqq \setst{x \in
    \Cone}{\Acal(x) = b}\). Suppose that \(\ConvexSet \cap
  \interior(\Cone) \neq \emptyset\) and that \(\ConvexSet\) is
  compact. Then the following are equivalent:
  \begin{enumerate}[(i)]
  \item for every \(c \in \EuclideanA^* \drop
    \polar{\sqbrac{\cone(\ConvexSet)}}\), the optimization problem
    \(\max\setst{\iprod{c}{x}}{\Acal(x) = b,\, x \in \Cone}\) and its
    dual have a strictly complementary pair of optimal solutions;
  \item \(\polar{\ConvexSet}\) is facially exposed.
  \end{enumerate}
\end{theorem}

\begin{proof}
  We start with the forward implication. Let \(\Face\) be a face
  of~\(\polar{\ConvexSet}\) such that \(\emptyset \neq \Face \neq
  \polar{\ConvexSet}\). Let \(c \in \relint(\Face)\). Since
  \(\ConvexSet\) is compact, \(\polar{\ConvexSet}\) has nonempty
  interior. Now \(\Face \neq \polar{\ConvexSet}\) implies that \(c \in
  \Face \subseteq \bd(\polar{\ConvexSet})\). Hence, \(c\) lies in some
  exposed face of~\(\polar{\ConvexSet}\). Thus, by
  Theorem~\ref{thm:normal-cone-exposed-faces}, there is some \(\xb \in
  \ConvexSet\) such that \(c \in \Face_{\xb}\), using the notation
  from~\eqref{eq:normal-cone-exposed-faces}. Thus, \(\iprod{c}{\xb} =
  1\), which shows that \(c \not\in
  \polar{\sqbrac{\cone(\ConvexSet)}}\). By hypothesis,
  \(\max\setst{\iprod{c}{x}}{\Acal(x) = b,\, x \in \Cone}\) and its
  dual have a strictly complementary pair of optimal solutions, so
  that \(c \in \relint\paren{\NormalCone{\ConvexSet}{\xh}}\) for some
  \(\xh \in \ConvexSet\) by Proposition~\ref{prop:4.1}. Note that \(1
  = \iprod{c}{\xb} \leq \iprod{c}{\xh}\) and \(c \in
  \polar{\ConvexSet}\) so \(\iprod{c}{\xh} = 1\). Thus, we find by
  Theorem~\ref{thm:normal-cone-exposed-faces} that \(c \in
  \relint(\Face_{\xh})\). But this means that \(\Face = \Face_{\xh}\),
  so that \(\Face\) is exposed.

  Suppose next that \(\polar{\ConvexSet}\) is facially exposed, and
  let \(c \in \EuclideanA^* \drop
  \polar{\sqbrac{\cone(\ConvexSet)}}\). Let \(\xb \in \argmax_{x \in
    \ConvexSet} \iprod{c}{x}\). Note that \(\iprod{c}{\xb} \leq 0\)
  would imply that \(c \in \polar{\sqbrac{\cone(\ConvexSet)}}\), so
  \(\iprod{c}{\xb} > 0\). Set \(\cb \coloneqq c/\iprod{c}{\xb}\) so
  that \(\iprod{\cb}{\xb} = 1\). Together with \(\cb \in
  \NormalCone{\ConvexSet}{\xb}\), this implies that \(\cb \in
  \Face_{\xb}\), using the notation
  from~\eqref{eq:normal-cone-exposed-faces}. By
  Theorem~\ref{thm:normal-cone-exposed-faces}, it follows that \(\cb\)
  lies in~\(\bd(\polar{\ConvexSet})\). Since \(\polar{\ConvexSet}\) is
  facially exposed, there exists an exposed face~\(\Face\)
  of~\(\polar{\ConvexSet}\) such that \(\cb \in \relint(\Face)\). By
  Theorem~\ref{thm:normal-cone-exposed-faces}, there exists \(\xh \in
  \ConvexSet\) such that \(\Face = \Face_{\xh}\). Thus,
  \eqref{eq:polar-face-relint} shows that \(\cb\) lies in
  \(\relint\paren{\NormalCone{\ConvexSet}{\xh}}\), and so does~\(c\).
  It follows from Proposition~\ref{prop:4.1} that
  \(\max\setst{\iprod{c}{x}}{\Acal(x) = b,\, x \in \Cone}\) and its
  dual have a strictly complementary pair of optimal solutions.
\end{proof}

One way to regard Theorem~\ref{thm:strict-compl-exposed-faces} is the
following. Determining directly whether
\(\max\setst{\iprod{c}{x}}{\Acal(x) = b,\, x \in \Cone}\) and its dual
have a pair of strictly complementary solutions individually for each
\(c \in \Euclidean^* \drop \polar{\sqbrac{\cone(\ConvexSet)}}\)
involves studying a small portion of the boundary of infinitely many
convex sets of the form~\(\setst{s \in \Cone^*}{s \in \Image(\Acal^*)
  - c}\), one for each objective vector~\(c\).
Theorem~\ref{thm:strict-compl-exposed-faces} offers, as an
alternative, determining the complete boundary structure of a single
convex set, namely, \(\polar{\ConvexSet}\).

In the same spirit as Proposition~\ref{prop:conic-lp-normal-cone}, the
polar of the feasible region of a linear conic optimization problem
may be described as follows (see, e.g.,
\cite[Remark~2.2]{LaurentP95a}):
\begin{proposition}
  \label{prop:conic-convex-set-polar}
  Let \(\Cone \subseteq \Euclidean\) be a pointed closed convex cone
  with nonempty interior. Let \(\Acal \ffrom \Euclidean \fto
  \EuclideanB^*\) be a linear function, and let \(b \in
  \EuclideanB^*\). Set \(\ConvexSet \coloneqq \setst{x \in
    \Cone}{\Acal(x) = b}\). Suppose that \(\ConvexSet \cap
  \interior(\Cone) \neq \emptyset\). Let \(\xb \in \EuclideanA\) such
  that \(\Acal(\xb) = b\). Then
  \begin{equation*}
    \polar{\ConvexSet}
    =
    \paren*{
      \Image(\Acal^*)
      \cap
      \polar{\set{\xb}}
    }
    -
    \Cone^*.
  \end{equation*}
\end{proposition}

\begin{proof}
  By the Strong Duality Theorem, membership of \(c \in
  \polar{\ConvexSet}\) is equivalent to the existence of \(y \in
  \EuclideanB\) and \(s \in \Cone^*\) such that \(\Acal^*(y) - s = c\)
  and \(\iprod{b}{y} \leq 1\). Note that \(\iprod{b}{y} =
  \iprod{\Acal(\xb)}{y} = \iprod{\xb}{\Acal^*(y)}\), so that \(c \in
  \polar{\ConvexSet}\) if and only if \(c \in \paren*{\Image(\Acal^*)
    \cap \polar{\set{\xb}}} - \Cone^*\).
\end{proof}

Let us apply this in the context of \MaxCut. Let \(G = (V,E)\) be a
graph, and let \(w \in \Reals_+^E\). Let \(\Laplacian{G}(w) \coloneqq
\sum_{ij \in E} w_{ij} \oprodsym{\paren{e_i-e_j}}\) denote the
weighted Laplacian of~\(G\) with respect to~\(w\). The SDP relaxation
for \MaxCut{} used by Goemans and Williamson~\cite{GoemansW95a} is
\begin{equation*}
  \max\setst{
    \iprod{\tfrac{1}{4}\Laplacian{G}(w)}{X}
  }{
    X \in \Elliptope{V}
  }.
\end{equation*}
Note that \(\Laplacian{G}(w) \succeq 0\). Moreover, \(\Psd{V} \cap
\polar{\sqbrac{\cone\paren{\Elliptope{V}}}} = \set{0}\), since \(I \in
\Elliptope{V}\). Thus, Theorem~\ref{thm:strict-compl-exposed-faces}
and Proposition~\ref{prop:conic-convex-set-polar} yield a concrete
approach to prove strict complementarity for all the ``relevant''
objective functions for the \MaxCut{} SDP. Namely, it suffices to
prove that \(\polar{\Elliptope{V}} = \paren{\Image(\Diag) \cap
  \polar{\set{I}}} - \Psd{V}\) is facially exposed.

\appendix
\section{Proof of Proposition~\ref{prop:compact-gauge-closed}}

\begin{proof}
  We may assume that \(\ConvexSet \neq \emptyset\), so that
  \(\gauge{\ConvexSet}{\cdot}\) is proper. Thus, we have to show that
  \(\gauge{\ConvexSet}{\cdot}\) is lower semi-continuous, i.e., for
  each \(\alpha \in \Reals\), the sub-level set \(S_{\alpha} \coloneqq
  \setst{x \in \Euclidean}{\gauge{\ConvexSet}{x} \leq \alpha}\) is
  closed. This is clearly the case for \(\alpha < 0\). For \(\alpha
  \geq 0\), we shall the fact that
  \begin{equation}
    \label{eq:1}
    \text{%
      there exists \(M > 0\) such that \(\gauge{\ConvexSet}{x} \geq
      \norm{x}/M\) for all \(x \in \Euclidean\).
    }
  \end{equation}
  Indeed, set \(M \coloneqq \max\setst{\norm{x}}{x \in \ConvexSet} +
  1\). Then, if \(\lambda \geq 0\) is such that \(x \in \lambda
  \ConvexSet\), we have \(\norm{x} \leq \lambda M\). This
  proves~\eqref{eq:1}.

  Note that~\eqref{eq:1} implies that \(S_0 = \set{0}\). Next, let
  \(\alpha > 0\). Let \((x_n)_{n \in \Naturals}\) be a sequence in
  \(S_{\alpha}\) converging to~\(x\). If \(x \neq 0\), then \(x \in
  S_0 \subseteq S_{\alpha}\), so assume that \(x \neq 0\). By taking
  some tail of the sequence, we may assume that \(\norm{x_n} \geq
  \tfrac{1}{2} \norm{x}\) for each~\(n\). Set \(\gamma_n \coloneqq
  \gauge{\ConvexSet}{x_n}\), so that \(\tfrac{\norm{x}}{2M} \leq
  \tfrac{\norm{x_n}}{M} \leq \gamma_n \leq \alpha\) for each~\(n\).
  Refine the sequence so that \(\gamma_n \to \gamma\) for some
  \(\gamma \in \Reals\) with \(\tfrac{\norm{x}}{2M} \leq \gamma \leq
  \alpha\). For each \(n\), there exists \(0 \leq \eps_n \leq 1/n\)
  such that \(x_n \in (\gamma_n + \eps_n) \ConvexSet\), whence
  \begin{equation*}
    \frac{1}{\gamma_n + \eps_n} x_n \in \ConvexSet
    \quad
    \forall n \in \Naturals.
  \end{equation*}
  By taking \(n \to \infty\), we find that \(\tfrac{1}{\gamma} x \in
  \ConvexSet\), so that \(\gauge{\ConvexSet}{x} \leq \gamma \leq
  \alpha\) and \(x \in S_{\alpha}\).
\end{proof}

\section{Proof of Theorem~\ref{thm:normal-cone-exposed-faces}}

\begin{proof}
  It is easy to check that
  \begin{equation}
    \label{eq:2}
    \Face_{\xb}
    =
    \setst{
      y \in \polar{\ConvexSet}
    }{
      \iprod{y}{\xb} = 1
    }
    \qquad
    \forall \xb \in \ConvexSet.
  \end{equation}
  Note that every set of the form \(\setst{y \in
    \polar{\ConvexSet}}{\iprod{y}{\xb} = 1}\) for some \(\xb \in
  \ConvexSet\) is an exposed face of~\(\ConvexSet\). Furthermore \(0
  \in \polar{\ConvexSet}\) shows that \(\Face_{\xb}\) is a proper
  subset of~\(\polar{\ConvexSet}\).

  Next let \(\Face\) be a nonempty exposed face
  of~\(\polar{\ConvexSet}\) other than \(\polar{\ConvexSet}\) itself,
  so that
  \begin{equation*}
    \Face
    =
    \setst*{
      y \in \polar{\ConvexSet}
    }{
      \iprod{y}{\zb}
      =
      \suppf{\polar{\ConvexSet}}{\zb}
    }
  \end{equation*}
  for some \(\zb \in \Euclidean\) such that
  \(\suppf{\polar{\ConvexSet}}{\zb} < \infty\). Then \(\Face \neq
  \polar{\ConvexSet}\) shows that \(\zb \neq 0\). By
  Proposition~\ref{prop:compact-gauge-closed}, we have
  \(\gauge{\ConvexSet}{x} = \suppf{\polar{\ConvexSet}}{x}\) for every
  \(x \in \Euclidean\); see, e.g., \cite[p.~125]{Rockafellar97a}.
  Thus, \(0 < \gauge{\ConvexSet}{\zb} =
  \suppf{\polar{\ConvexSet}}{\zb} < \infty\). Set \(\xb \coloneqq
  \zb/\gauge{\ConvexSet}{\zb}\), so that \(\gauge{\ConvexSet}{\xb} =
  1\). Then
  \begin{equation*}
    \Face
    =
    \setst*{
      y \in \polar{\ConvexSet}
    }{
      \iprod{y}{\xb}
      =
      1
    }
    =
    \Face_{\xb}
  \end{equation*}
  by~\eqref{eq:2}. This completes the precise description of all the
  nonempty exposed faces of~\(\polar{\ConvexSet}\).

  Finally, let \(\xb \in \ConvexSet\) such that \(\Face_{\xb} \neq
  \emptyset\). To prove~\eqref{eq:polar-face-relint}, it suffices to
  prove that \(\relint\paren*{\NormalCone{\ConvexSet}{\xb}}\) meets
  \(\setst{y \in \Euclidean^*}{\iprod{y}{\xb} = 1}\); see, e.g.,
  \cite[Theorem~6.5]{Rockafellar97a}. Suppose not. Then there is a
  hyperplane separating \(\NormalCone{\ConvexSet}{\xb}\) and the
  second set, i.e., there exists a nonzero \(h \in \Euclidean\) and
  \(\alpha \in \Reals\) such that
  \begin{gather}
    \NormalCone{\ConvexSet}{\xb}
    \subseteq
    \setst{
      y \in \Euclidean^*
    }{
      \iprod{y}{h}
      \leq
      \alpha
    },
    \\
    \setst{
      y \in \Euclidean^*
    }{
      \iprod{y}{\xb} = 1
    }
    \subseteq
    \setst{
      y \in \Euclidean^*
    }{
      \iprod{y}{h} \geq \alpha
    }.
  \end{gather}
  Note that \(0 \in \NormalCone{\ConvexSet}{\xb}\) shows that \(\alpha
  \geq 0\), and that \(h = \lambda \xb\) for some \(\lambda > 0\).
  Positive homogeneity now shows that \(\iprod{y}{h} \leq 0\) for all
  \(y \in \NormalCone{\ConvexSet}{\xb}\), whence \(\Face_{\xb} =
  \emptyset\), a contradiction.
\end{proof}


\begin{thebibliography}{10}

\bibitem{Alfakih06a}
A.~Y. Alfakih.
\newblock A remark on the faces of the cone of {E}uclidean distance matrices.
\newblock {\em Linear Algebra Appl.}, 414(1):266--270, 2006.

\bibitem{AlizadehHO97a}
F.~Alizadeh, J.-P.~A. Haeberly, and M.~L. Overton.
\newblock Complementarity and nondegeneracy in semidefinite programming.
\newblock {\em Math. Programming}, 77(2, Ser. B):111--128, 1997.
\newblock Semidefinite programming.

\bibitem{Barvinok01a}
A.~Barvinok.
\newblock A remark on the rank of positive semidefinite matrices subject to
  affine constraints.
\newblock {\em Discrete Comput. Geom.}, 25(1):23--31, 2001.

\bibitem{Barvinok95a}
A.~I. Barvinok.
\newblock Problems of distance geometry and convex properties of quadratic
  maps.
\newblock {\em Discrete Comput. Geom.}, 13(2):189--202, 1995.

\bibitem{BhardwajRS11a}
A.~Bhardwaj, P.~Rostalski, and R.~Sanyal.
\newblock Deciding polyhedrality of spectrahedra.
\newblock \url{http://arxiv.org/abs/1102.4367}, December 2012.
\newblock ArXiv e-prints, math.OC/1102.4367.

\bibitem{BolteDL11a}
J.~Bolte, A.~Daniilidis, and A.~S. Lewis.
\newblock Generic optimality conditions for semialgebraic convex programs.
\newblock {\em Math. Oper. Res.}, 36(1):55--70, 2011.

\bibitem{ChristensenV79a}
J.~P.~R. Christensen and J.~Vesterstr{\o}m.
\newblock A note on extreme positive definite matrices.
\newblock {\em Math. Ann.}, 244(1):65--68, 1979.

\bibitem{ChuaT08a}
C.~B. Chua and L.~Tun{\c{c}}el.
\newblock Invariance and efficiency of convex representations.
\newblock {\em Math. Program.}, 111(1-2, Ser. B):113--140, 2008.

\bibitem{Silva13a}
M.~K. de~Carli~Silva.
\newblock {\em Geometric Ramifications of the Lov{\'a}sz Theta Function and
  Their Interplay with Duality}.
\newblock PhD thesis, Department of Combinatorics and Optimization, University
  of Waterloo, 2013.

\bibitem{DeSimone90a}
C.~De~Simone.
\newblock The cut polytope and the {B}oolean quadric polytope.
\newblock {\em Discrete Math.}, 79(1):71--75, 1989/90.

\bibitem{DrusvyatskiyL11a}
D.~Drusvyatskiy and A.~S. Lewis.
\newblock Generic nondegeneracy in convex optimization.
\newblock {\em Proc. Amer. Math. Soc.}, 139(7):2519--2527, 2011.

\bibitem{GoemansW95a}
M.~X. Goemans and D.~P. Williamson.
\newblock Improved approximation algorithms for maximum cut and satisfiability
  problems using semidefinite programming.
\newblock {\em J. Assoc. Comput. Mach.}, 42(6):1115--1145, 1995.

\bibitem{GortlerT10a}
S.~J. Gortler and D.~P. Thurston.
\newblock Characterizing the universal rigidity of generic frameworks.
\newblock {\em ArXiv e-prints}, December 2010.

\bibitem{GroetschelLS84a}
M.~Gr{\"o}tschel, L.~Lov{\'a}sz, and A.~Schrijver.
\newblock Polynomial algorithms for perfect graphs.
\newblock In {\em Topics on perfect graphs}, volume~88 of {\em North-Holland
  Math. Stud.}, pages 325--356. North-Holland, Amsterdam, 1984.

\bibitem{GroetschelLS86a}
M.~Gr{\"o}tschel, L.~Lov{\'a}sz, and A.~Schrijver.
\newblock Relaxations of vertex packing.
\newblock {\em J. Combin. Theory Ser. B}, 40(3):330--343, 1986.

\bibitem{GvozdenovicL08a}
N.~Gvozdenovi{\'c} and M.~Laurent.
\newblock The operator {$\Psi$} for the chromatic number of a graph.
\newblock {\em SIAM J. Optim.}, 19(2):572--591, 2008.

\bibitem{KleinbergG98a}
J.~Kleinberg and M.~X. Goemans.
\newblock The {L}ov\'asz theta function and a semidefinite programming
  relaxation of vertex cover.
\newblock {\em SIAM J. Discrete Math.}, 11(2):196--204 (electronic), 1998.

\bibitem{Knuth94a}
D.~E. Knuth.
\newblock The sandwich theorem.
\newblock {\em Electron. J. Combin.}, 1:Article 1, approx.\ 48 pp.\
  (electronic), 1994.

\bibitem{Larman71a}
D.~G. Larman.
\newblock On a conjecture of {K}lee and {M}artin for convex bodies.
\newblock {\em Proc. London Math. Soc. (3)}, 23:668--682, 1971.

\bibitem{LaurentP95a}
M.~Laurent and S.~Poljak.
\newblock On a positive semidefinite relaxation of the cut polytope.
\newblock {\em Linear Algebra Appl.}, 223/224:439--461, 1995.
\newblock Special issue honoring Miroslav Fiedler and Vlastimil Pt{\'a}k.

\bibitem{LaurentP96a}
M.~Laurent and S.~Poljak.
\newblock On the facial structure of the set of correlation matrices.
\newblock {\em SIAM J. Matrix Anal. Appl.}, 17(3):530--547, 1996.

\bibitem{LaurentPR97a}
M.~Laurent, S.~Poljak, and F.~Rendl.
\newblock Connections between semidefinite relaxations of the max-cut and
  stable set problems.
\newblock {\em Math. Programming}, 77(2, Ser. B):225--246, 1997.
\newblock Semidefinite programming.

\bibitem{LiT94a}
C.-K. Li and B.~S. Tam.
\newblock A note on extreme correlation matrices.
\newblock {\em SIAM J. Matrix Anal. Appl.}, 15(3):903--908, 1994.

\bibitem{Loewy80a}
R.~Loewy.
\newblock Extreme points of a convex subset of the cone of positive
  semidefinite matrices.
\newblock {\em Math. Ann.}, 253(3):227--232, 1980.

\bibitem{Lovasz79a}
L.~Lov{\'a}sz.
\newblock On the {S}hannon capacity of a graph.
\newblock {\em IEEE Trans. Inform. Theory}, 25(1):1--7, 1979.

\bibitem{LovaszS91a}
L.~Lov{\'a}sz and A.~Schrijver.
\newblock Cones of matrices and set-functions and {$0$}-{$1$} optimization.
\newblock {\em SIAM J. Optim.}, 1(2):166--190, 1991.

\bibitem{McElieceRR78a}
R.~J. McEliece, E.~R. Rodemich, and H.~C. Rumsey, Jr.
\newblock The {L}ov\'asz bound and some generalizations.
\newblock {\em J. Combin. Inform. System Sci.}, 3(3):134--152, 1978.

\bibitem{NieRS10a}
J.~Nie, K.~Ranestad, and B.~Sturmfels.
\newblock The algebraic degree of semidefinite programming.
\newblock {\em Math. Program.}, 122(2, Ser. A):379--405, 2010.

\bibitem{Pataki96a}
G.~Pataki.
\newblock Cone-{LP}'s and semidefinite programs: geometry and a simplex-type
  method.
\newblock In {\em Integer programming and combinatorial optimization
  ({V}ancouver, {BC}, 1996)}, volume 1084 of {\em Lecture Notes in Comput.
  Sci.}, pages 162--174. Springer, Berlin, 1996.

\bibitem{Pataki98a}
G.~Pataki.
\newblock On the rank of extreme matrices in semidefinite programs and the
  multiplicity of optimal eigenvalues.
\newblock {\em Math. Oper. Res.}, 23(2):339--358, 1998.

\bibitem{PatakiT01a}
G.~Pataki and L.~Tun{\c{c}}el.
\newblock On the generic properties of convex optimization problems in conic
  form.
\newblock {\em Math. Program.}, 89(3, Ser. A):449--457, 2001.

\bibitem{Rockafellar97a}
R.~T. Rockafellar.
\newblock {\em Convex analysis}.
\newblock Princeton Landmarks in Mathematics. Princeton University Press,
  Princeton, NJ, 1997.
\newblock Reprint of the 1970 original, Princeton Paperbacks.

\bibitem{SaundersonCPW12a}
J.~Saunderson, V.~Chandrasekaran, P.~A. Parrilo, and A.~S. Willsky.
\newblock Diagonal and low-rank matrix decompositions, correlation matrices,
  and ellipsoid fitting.
\newblock {\em ArXiv e-prints}, April 2012.

\bibitem{Schrijver79a}
A.~Schrijver.
\newblock A comparison of the {D}elsarte and {L}ov\'asz bounds.
\newblock {\em IEEE Trans. Inform. Theory}, 25(4):425--429, 1979.

\bibitem{Schrijver03b}
A.~Schrijver.
\newblock {\em Combinatorial optimization. {P}olyhedra and efficiency. {V}ol.
  {B}}, volume~24 of {\em Algorithms and Combinatorics}.
\newblock Springer-Verlag, Berlin, 2003.
\newblock Matroids, trees, stable sets, Chapters 39--69.

\bibitem{Shepherd01a}
F.~B. Shepherd.
\newblock The theta body and imperfection.
\newblock In {\em Perfect graphs}, Wiley-Intersci. Ser. Discrete Math. Optim.,
  pages 261--291. Wiley, Chichester, 2001.

\bibitem{Szegedy94a}
M.~Szegedy.
\newblock A note on the theta number of {L}ov{\'a}sz and the generalized
  {D}elsarte bound.
\newblock In {\em Proceedings of the 35th {A}nnual {IEEE} {S}ymposium on
  {F}oundations of {C}omputer {S}cience}, 1994.

\bibitem{WolkowiczSV00a}
H.~Wolkowicz, R.~Saigal, and L.~Vandenberghe, editors.
\newblock {\em Handbook of semidefinite programming}.
\newblock International Series in Operations Research \& Management Science,
  27. Kluwer Academic Publishers, Boston, MA, 2000.
\newblock Theory, algorithms, and applications.

\end{thebibliography}
\end{document}